\documentclass[11pt,reqno]{amsart}
\usepackage{prettyref}
\usepackage{amstext}
\usepackage{amsthm,amsmath,mathrsfs}
\usepackage{amssymb,xcolor}
\usepackage[unicode=true,pdfusetitle,
 bookmarks=true,bookmarksnumbered=false,bookmarksopen=false,
 breaklinks=false,pdfborder={0 0 1},backref=false,colorlinks=false]
 {hyperref}
 \usepackage[normalem]{ulem}
 
 \usepackage{graphicx}

\makeatletter
\numberwithin{equation}{section}
\theoremstyle{plain}
\newtheorem{thm}{\protect\theoremname}[section]
  \theoremstyle{definition}
  \newtheorem{defn}[thm]{\protect\definitionname}
  \theoremstyle{plain}
  \newtheorem{prop}[thm]{\protect\propositionname}
  \theoremstyle{plain}
  \newtheorem{lem}[thm]{\protect\lemmaname}
    \newtheorem*{lem*}{\protect\lemmaname}
    \newtheorem*{thm*}{\protect\theoremname}
  \newtheorem{cor}[thm]{Corollary}
\theoremstyle{definition}
\newtheorem{rem}[thm]{Remark}
\usepackage{mathrsfs,amsthm,amssymb,bbm,fullpage}

\renewcommand{\liminf}{\varliminf}
\renewcommand{\limsup}{\varlimsup}

\newcommand{\cA}{\mathcal{A}}
\newcommand{\sA}{\mathscr{A}}

\newcommand{\cM}{\mathscr{M}}
\newcommand{\cN}{\mathcal{N}}
\newcommand{\cE}{\mathcal{E}}
\newcommand{\cF}{\mathcal{F}}
\newcommand{\cG}{\mathcal{G}}

\newcommand{\cX}{\mathcal{X}}

\newcommand{\cB}{\mathcal{B}}
\newcommand{\cP}{\mathcal{P}}

\newcommand{\cR}{\mathcal{R}}
\newcommand{\sR}{\mathscr{R}}
\newcommand{\cS}{\mathcal{S}}
\newcommand{\R}{\mathbb{R}}

\newcommand{\N}{\mathbb{N}}

\newcommand{\prob}{\mathbb{P}}
\newcommand{\E}{\mathbb{E}}

\newcommand{\indicator}[1]{\mathbbm{1}_{#1}}

\newcommand{\eqdist}{\stackrel{(d)}{=}}
\newcommand{\tensor}{\otimes}

\newcommand{\bLambda}{\mathbf{\Lambda}}

\newcommand{\abs}[1]{\lvert#1\rvert}
\newcommand{\norm}[1]{\lvert\lvert#1\rvert\rvert}

\newcommand{\supp}{\text{supp}}

\newcommand{\Var}{\text{Var}}
\newcommand{\gibbs}[1]{\left\langle#1\right\rangle}

\newcommand{\ceil}[1]{\lceil#1\rceil}

\makeatother

  \providecommand{\definitionname}{Definition}
  \providecommand{\lemmaname}{Lemma}
  \providecommand{\propositionname}{Proposition}
\providecommand{\theoremname}{Theorem}
\newrefformat{prop}{Proposition \ref{#1}}

\begin{document}
\title{The overlap gap property in principal submatrix recovery}
\author{David Gamarnik}
\address{Sloan School of Management, Massachusetts Institute of Technology, \tt{gamarnik@mit.edu}}

\author{Aukosh Jagannath}
\address{Department of Statistics and Actuarial Sciences and Department of Applied Mathematics, University of Waterloo,  \tt{a.jagannath@uwaterloo.ca}}

\author{Subhabrata Sen} 
\address{Department of Statistics, Harvard University \tt{subhabratasen@fas.harvard.edu}} 


\date{\today}
\maketitle
\begin{abstract}
We study support recovery for a $k \times k$ principal submatrix with elevated mean $\lambda/N$, hidden in an $N\times N$ symmetric mean zero Gaussian matrix. Here $\lambda>0$ is a universal constant, and we assume $k = N \rho$ for some  constant $\rho \in (0,1)$. We establish that {there exists a constant $C>0$ such that} the MLE recovers a constant proportion of the hidden submatrix if $\lambda {\geq C} \sqrt{\frac{1}{\rho} \log \frac{1}{\rho}}$, {while such recovery is information theoretically impossible if $\lambda = o( \sqrt{\frac{1}{\rho} \log \frac{1}{\rho}} )$}. The MLE is computationally intractable in general, and in fact, for $\rho>0$ sufficiently small, this problem is conjectured to exhibit a \emph{statistical-computational gap}. To provide rigorous evidence for this,  we study the likelihood landscape for this problem, and establish that for some $\varepsilon>0$ and  $\sqrt{\frac{1}{\rho} \log \frac{1}{\rho} } \ll \lambda \ll \frac{1}{\rho^{1/2 + \varepsilon}}$, the problem exhibits a variant of the \emph{Overlap-Gap-Property (OGP)}. As a direct consequence, we establish that a family of local MCMC based algorithms do not achieve optimal recovery. Finally, we establish that for $\lambda > 1/\rho$, a simple spectral method recovers a constant proportion of the hidden submatrix. 
 \end{abstract}

\section{Introduction}
In this paper, we study support recovery for a planted principal submatrix in a large symmetric Gaussian matrix. Formally, we observe a symmetric matrix $A = (A_{ij}) \in \mathbb{R}^{N \times N}$, 
\begin{align}
A_{ij} = \theta_{ij} + W_{ij}. \label{eq:model_seq}
\end{align}
Throughout, we assume that $W$ is a GOE random matrix; in other words, $\{ W_{ij} : i \leq j \}$ are independent Gaussian random variables, with $\{W_{ij} : i < j \}$ i.i.d. $\mathcal{N}(0,1/N)$, and $\{W_{ii}: 1 \leq i \leq N\}$ i.i.d. $\mathcal{N}(0, 2/N)$. Regarding the mean matrix $\mathbf{\Theta} = (\theta_{ij})$, we assume that there exists $U \subset [N] := \{1, 2, \cdots, N\}$, $|U| =k$, such that 
\begin{align}
\theta_{ij} = \begin{cases}
\frac{\lambda}{N} & \textrm{if}\,\, i,j \in U \\
0 & \textrm{o.w.},
\end{cases} \nonumber 
\end{align}
where $\lambda>0$ is a constant independent of $N$. Equivalently, the observed matrix $A$ may be re-written as 
\begin{align}
A=\frac{\lambda}{N}vv^{T}+W,\label{eq:model}
\end{align}
where $v = (v_i)  \in \{0,1\}^N$, with $\sum_i v_i = k$. Throughout the subsequent discussion, we will denote the set of such boolean vectors as $\Sigma_N(\frac{k}{N})$. 

In the setting introduced above, the following statistical questions are natural. 
\begin{enumerate}
\item (Detection) Can we detect the presence of the planted submatrix, i.e., can we consistently test
\begin{align}
\mathrm{H}_0: \lambda =0 \,\,\,\,\, \mathrm{vs.} \,\,\,\,\,\, \mathrm{H}_1: \lambda>0. \nonumber
\end{align}
\item (Recovery) How large should $\lambda$ be, such that the support of $v$ can be recovered \emph{approximately}?
\item (Efficient Recovery) When can support recovery be accomplished using a computationally feasible procedure? 
\end{enumerate}

Here, we study support recovery in the special case $k = N\rho$, for some $\rho \in (0,1)$. To ensure that this problem is well-defined for all $N$, we work along a sequence $\rho_N \to \rho$ 
such that $N\rho_N \in \mathbb{N}$ and $N\rho\in (N\rho_N-1/2,N\rho_N+1/2]$. Note that in this case, the corresponding submatrix detection question \cite{Butucea2013submatrix} is trivial, and a test which rejects for large values of the sum of the matrix entries consistently detects the underlying submatrix for any $\lambda >0$. Motivated by the sparse PCA problem, we will study support recovery in this setup in the double limit $\rho \to 0$, following $N \to \infty$. Deshpande-Montanari \cite{deshpande2014pca} initiated a study of the problem \eqref{eq:model}, and established that Bayes optimal recovery of the matrix $\mathbf{\Theta}$  can be accomplished using an Approximate Message Passing based algorithm, whenever $\rho >0$ is sufficiently large (specifically, $\rho > 0.041$).  In \cite{lesieur2015spca, lesieur2017lowrank}, the authors analyze optimal Bayesian estimation in the $\rho\to 0$ regime, and based on the behavior of the fixed points of a state-evolution system, conjecture the existence of an algorithmically hard phase in this problem; specifically, they conjecture that the minimum signal $\lambda$ required for accurate support recovery using feasible algorithms should be significantly higher than the information theoretic minimum. This conjecture has been repeatedly quoted in various follow up works  \cite{Dia2016replica,krzakala2016rankone,lelarge2019symmetric}, but to the best of our knowledge, it has not been rigorously established in the prior literature. In this paper, we study the likelihood landscape of this problem, and provide rigorous evidence to the veracity of this conjecture. 

From a purely conceptual viewpoint, the existence of a computationally hard phase in problem \eqref{eq:model} is particularly striking. 
In the context of rank one matrix estimation contaminated with additive Gaussian noise
\eqref{eq:model}, it is known that if the spike $v$ is sampled uniformly at random from the unit sphere, PCA recovers the underlying signal, whenever its detection is possible \cite{montanari2015spectral}. In contrast, for rank one tensor estimation under additive Gaussian noise \cite{montanari2014tensor}, there exists an extensive gap between the threshold for detection \cite{JLM18}, and the threshold where tractable algorithms are successful \cite{BGJ2018,hopkins2016fast, wein2019tensor}. Thus at first glance, the matrix and tensor problems appear to behave very differently. However, as the present paper establishes, once the planted spike is sufficiently sparse, a hard phase re-appears in the matrix problem.

This problem has natural connections to the planted clique problem \cite{alon1998clique}, sparse PCA  \cite{Amini2008spca}, biclustering \cite{bhamidi2012, gamarnik2016, shabalin2009submatrix}, and community detection \cite{Abbe2017SBM, montanari2015community, moore2017community}. All these problems are expected to exhibit a statistical-computational gap---there are regimes where optimal statistical performance might be impossible to achieve using computationally feasible statistical procedures. The potential existence of such fundamental computational barriers has attracted significant attention recently in Statistics, Machine Learning, and Theoretical Computer Science. A number of widely distinct approaches have been used to understand this phenomenon better---prominent ones include average case reductions \cite{Berthet2013spca, Brennan2018sparse, Brennan2019submatrix, cai2017submatrix,gao2017scca, Ma2015submatrix}, convex relaxations \cite{barak2019clique, chandra2013convex, deshpande2015sos, hopkins2016sos,ma2015spca,meka2015SOS}, query lower bounds \cite{feldman2017planted, rossman2010cliques}, and direct analysis of specific algorithms \cite{bala2011submatrix,chen2016planted, kolar2011localization}. The submatrix recovery problem itself has been investigated from a number of distinct perspectives. We defer an in-depth discussion of these approaches to the end of the Introduction.

In comparison to the approaches originating from Computer Science and Optimization,
a completely different perspective to this problem comes from statistical physics, particularly the study of spin glasses. This approach seeks to understand algorithmic hardness in random optimization problems as a consequence of the underlying geometry of the problem--- specifically, the structure of the near optimizers. The Overlap Gap property (OGP) has emerged as a recurrent theme in this context (for an illustration, see Fig \ref{defn:ogp}). 
 At a high level, the Overlap Gap Property (OGP) looks at approximate optimizers in a problem, and establishes that any two such approximate optimizers must either be close to each other, or far away. In other words, the one-dimensional set of distances between the near optimal states is disconnected. This property has been established for various problems arising from theoretical computer science and combinatorial optimization, for instance random constraint satisfaction \cite{achlioptas2011ogp,gamarnik2017naesat,mezard05geometry}, Max Independent Set \cite{gamarnik2016, rahman2017ind}, and a maxcut problem on hypergraphs \cite{chen2019maxcut}. Further, OGP has been shown to act as a barrier to the success of a family of ``local algorithms" on sparse random graphs \cite{cojaoghlan2017walksat,chen2019maxcut,gamarnik2016,gamarnik2017naesat}.  This perspective has been introduced to the study of inference problems arising in Statistics and Machine Learning \cite{BGJ2018,gamarnik2016, gamarnik2017regression1, gamarnik2017regression2, gamarnik2019planted} in recent works by the first two authors which has yielded new insights into algorithmic barriers in these problems. As an aside, we note that exciting new developments in the study of mean field spin glasses \cite{ABM2018CRM,montanari2018sk,subag2018fRSB}, establish that in certain problems without OGP,
 it is actually possible to obtain approximate optimizers using appropriately designed polynomial time algorithms. This lends further credence to the belief that OGP captures a fundamental algorithmic barrier in random optimization problems--- understanding this phenomenon in greater depth remains an exciting direction for future research.

\subsection{Results}
\label{subsec:results}
We initiate the study of \emph{approximate support recovery} in the setting of \eqref{eq:model}. To introduce this notion, let us begin with the following definitions and observations. For $v \in \{0,1\}^N$, 
define the \emph{support} of $v$ to be the subset $S(v)\subset[N]$, such that 
\[
S(v) = \{ k \in [N] : v_k =1\}.
\]
To estimate the support, it is evidently equivalent to 
produce an estimator $\hat{v}$ that takes values in the Boolean hypercube $\{0,1\}^N$.
Observe that if $\hat{v}\in \Sigma_N(\rho)$ is drawn uniformly at random, then 
the intersection of the support of $\hat{v}$ and $v$ satisfies
\[
\frac{1}{N}\abs{S(v)\cap S(\hat{v})} = \frac{1}{N} (\hat{v}, v) \to \rho^2,
\]
where $(\cdot,\cdot)$ denotes the usual Euclidean inner product.
For an estimator $\hat v$ of $v$, in the following, we call the normalized inner product $(\hat v,v)/N$ the \emph{overlap}
of the estimator with $v$, or simply the overlap. We are interested in the statistical and algorithmic feasibility
of recovering a non-trivial fraction of the support. To this end, we introduce the notion of approximate recovery, 
which is defined as follows. 
\begin{defn}[Approximate Recovery]
A sequence of  estimators $\hat{v}_N = \hat{v}_N(A) \in \Sigma_N(\rho_N)$ is said to recover the support of $v$ approximately if there exists $c>0$, independent of $N, \rho$, such that $\langle v , \hat{v} \rangle > c \rho N$ with high probability as $N\to \infty$. 
\end{defn}
\noindent
Observe that since $\hat{v},v\in\Sigma_N(\rho_N)$, $\hat{v}$ recovers the support approximately 
if and only if it  achieves a non-trivial Hamming loss, i.e.,  $d_H(v,\hat{v}) \leq 2N\rho(1-c)$ for some $c>0$ with high probability. (Here $d_H$ is the Hamming distance).

We study here the question of approximate support recovery in this context, and exhibit a regime of parameters $(\lambda, \rho)$ where this problem exhibits OGP. This provides rigorous evidence to the existence of a computationally hard phase in this problem, as suggested in \cite{lesieur2015spca, lesieur2017lowrank}. To substantiate this claim, we establish that OGP acts as a barrier to the success of certain Markov-chain based algorithms. Finally, we show that for very large signal strengths $\lambda$, approximate support recovery is easy, and can be achieved by rounding the largest eigenvector of $A$.

To state our results in detail, we first introduce the Maximum Likelihood Estimator (MLE) in this setting. Let
\begin{align}
H_{N}(x)= (x, Wx), \label{eq:hamiltonian}
\end{align}
and consider the MLE for $v$, 
\begin{align}
\hat{v}_{ML} & =\mathrm{argmax}_{x \in \Sigma_{N}(\rho_N)} (x,Ax)  =\displaystyle\mathrm{argmax}_{x\in\Sigma_{N}(\rho_N)} \Big\{ H_{N}(x)+\frac{\lambda}{N}(x,v)^{2} \Big\}.  \label{eq:ml_defn}
\end{align}
Our first result is information theoretic, and derives the minimum signal strength $\lambda$ required for approximate support recovery
\begin{thm}
\label{thm:small_signal_main}
 If $\lambda = o\Big( \sqrt{\frac{1}{\rho} \log \frac{1}{\rho}} \Big) $, approximate recovery is information theoretically impossible.
On the other hand, for any $\varepsilon>0$, if $\lambda > (2 + \varepsilon) \sqrt{\frac{1}{\rho} \log \frac{1}{\rho} }$, the MLE recovers the support approximately. 
\end{thm}
\noindent
Thus for any $\varepsilon >0$ and $\lambda > (2+\varepsilon) \sqrt{\frac{1}{\rho} \log \frac{1}{\rho} }$, there exists at least one estimator, namely, the MLE, which performs approximate recovery. However, the MLE is computationally intractable in general. Our next result analyzes the likelihood landscape for this problem, and establishes that the problem exhibits OGP for certain $(\lambda, \rho)$ parameters in this phase.  

To this end, we introduce a version of \emph{the overlap gap property} in this context. Consider the constrained maximum likelihood,
\begin{align}
E_{N}(q;\rho,\lambda)=\frac{1}{N}\max_{\substack{x\in\Sigma_{N}(\rho_N)\\
(x,v)=Nq
}
}H_{N}(x)+\lambda q^2, \label{eq:restricted}
\end{align}
which denotes the maximum likelihood subject to the additional constraint
of achieving overlap $q$ with $v$. 
For any $q \in [0, \rho]$, fix a sequence $q_N \to q$ such that  $Nq_N \in \mathbb{N}$ and $Nq \in [Nq_N - 1/2, Nq_N + 1/2)$.   
We establish in Theorem \ref{thm:var-form-gs} below that for all $q \in [0, \rho]$, we have that 
\begin{align}
E_{N}(q_N; \rho_N, \lambda) \to E(q; \rho, \lambda)\nonumber
\end{align} 
 as $N \to \infty$ and that $E(q;\rho,\lambda)$ can be computed
by a deterministic Parisi-type \cite{panchenkobook} variational problem. 
In the subsequent, we refer to $E(q;\rho,\lambda)$ as the \emph{constrained ground state
energy}, or simply \emph{the constrained energy}.
We are now in the position to define the overlap gap property. 

\begin{defn}
\label{defn:ogp}
For some $\varepsilon>0$, we say that the model \eqref{eq:model} with sparsity $\rho$ and signal-to-noise
ratio $\lambda > 0$ exhibits the $\varepsilon$-\emph{overlap gap property ($\varepsilon$-OGP) } if there exist three points $w<x<y < \rho $ such that:
\begin{enumerate}
\item $w<\rho^2 < x$,  $ y < \rho^{1+ \varepsilon} $ , 
\item $\max\{E(w; \rho, \lambda) , E(y; \rho, \lambda)\} < E(x; \rho, \lambda)$, 
\item $\sup_{(w,\varepsilon \rho]}E(q;\rho,\lambda) < \sup_{[\varepsilon \rho,\rho]} E(q;\rho, \lambda)$. 
\end{enumerate}
\end{defn}
\begin{figure}
\begin{center}
\includegraphics[scale=0.7]{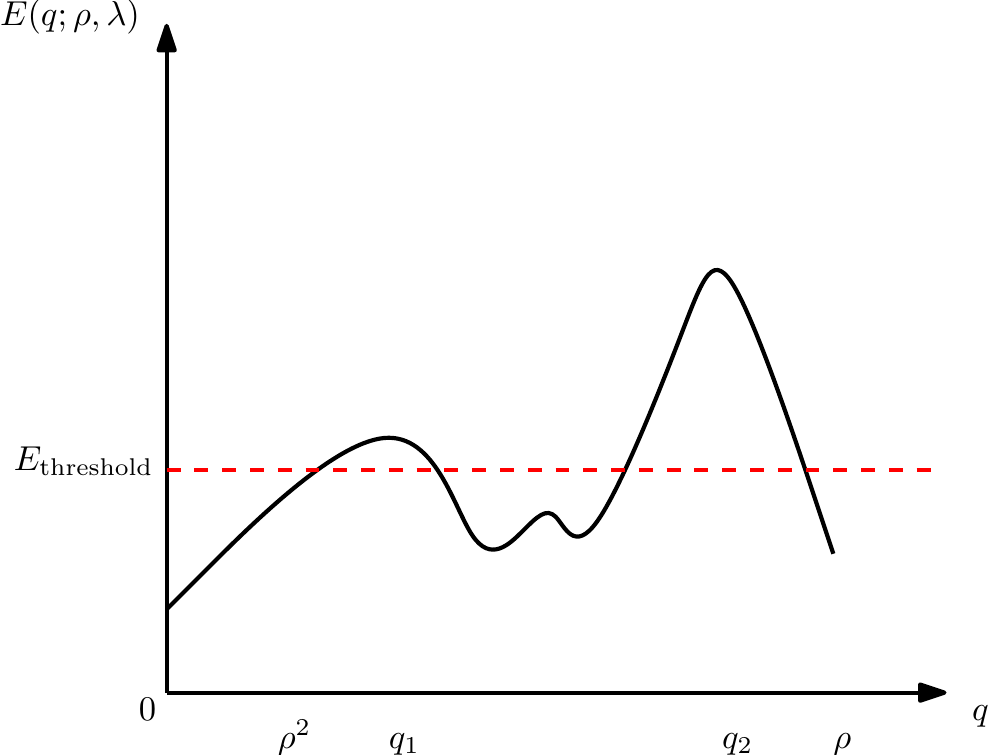}
\caption{An illustration of OGP. Observe that under the conditions of Definition \ref{defn:ogp},  there exists a certain threshold energy $E_{\mathrm{threshold}}$ such that the set of possible overlaps $\{ \frac{1}{N} (x,v) :x \in \Sigma_N(\rho), (x, A x) > N E_{\mathrm{threshold}}   \}$ has a gap.}
\end{center}
\label{fig1}
\end{figure}

Our main result establishes that the planted submatrix problem \eqref{eq:model} admits the 
overlap gap property in the limit of high sparsity and moderate signal-to-noise ratios.
\begin{thm}
\label{thm:main}
For any $\alpha < 2 -\sqrt{2} \approx 0.586$ and $C_1  >2$, there exist $\varepsilon>0$,  \sout{and $C_2$} such that for $\rho$ sufficiently small, for all $C_1 \sqrt{\frac{1}{\rho}\log\frac{1}{\rho}} < \lambda <  \left(\frac{1}{\rho}\right)^{\alpha}$, the planted sub-matrix problem has the  $\varepsilon$-overlap gap property. 
\end{thm}
Note that by \prettyref{thm:small_signal_main}, approximate recovery is possible in the entire regime covered by Theorem \ref{thm:main}; however, the likelihood-landscape exhibits OGP in this part of the parameter space. Put simply, the overlap gap property states that the MLE achieves an overlap $q_2$ that is substantially better
than $\rho^2$, but in the interval $[0,q_2]$, the constrained energy has another local maximum.
Heuristically, this suggests that a local optimization procedure, if initialized uniformly at random (and thus starting at overlap $\rho^2$), will get trapped at a local maximum $q_1$
which is sub-optimal as compared to the true global optimum in terms of both the likelihood and the overlap. We illustrate this notion visually in Figure \ref{fig1}.

 Observe that the hard phase becomes more prominent as $\rho \to 0$.

Finally, we establish that the OGP established above acts as a barrier to a family of local MCMC type algorithms. To this end, consider a Gibbs distribution on the configuration space $\Sigma_N(\rho_N)$ with 
\begin{align}
\pi_{\beta}(\{x\}) \propto \exp(\beta (x,Ax)), \nonumber 
\end{align}
for some $\beta>0$ and $A$ defined as in \eqref{eq:model}. Note that for any fixed $N$, as $\beta \to \infty$, the sample from $\pi_{\beta}$ approximates the MLE \ref{eq:ml_defn}. Thus a simple proxy for the MLE seeks to sample from the distribution $\pi_{\beta}$ for $\beta$ sufficiently large. It is natural to use local Markov chains to sample from this distribution. Specifically, construct a graph with vertices being the elements of $\Sigma_N(\rho_N)$, and add an edge between two states $x,x' \in \Sigma_N(\rho_N)$ if they are at Hamming distance $2$. Finally, let $Q_x$ denote a nearest-neighbor Markov chain on this graph started from $X_0=x$ that is reversible with respect to the stationary distribution $\pi_{\beta}$. 
The following theorem establishes hitting time bounds for any such Markov Chain.

\begin{thm}\label{thm:main-MCMC}
For $(\lambda,\rho)$ as in \prettyref{thm:main}, there exist points $a<\rho^{2}<b$
with $b/\rho\to 0$ as $\rho\to0$ and an $h>0$ such that for some $c>0$, with probability
$1-O(e^{-cN})$ , if $\mathcal{I}=(a,b)$, then the exit time of $\mathcal{I}$, $\tau_{\mathcal{I}^{c}}$,
satisfies
\[
\int Q_{x}(\tau_{\mathcal{I}^{c}}\leq T)d\pi_\beta(x\vert \mathcal{I})\leq Te^{-chN}.
\]
\end{thm}

\noindent The proof of this result immediately follows by combining \prettyref{thm:main} with Corollary \ref{cor:OGP-to-exit}. Thus OGP indeed acts as a barrier to these algorithms, and furnishes rigorous evidence of a hard phase in this problem.

As an aside, we also note that Theorem \ref{thm:var-form-gs} implies 
\begin{align}
\lim_{N \to \infty} \frac{1}{N} \max_{x \in \Sigma_N(\rho_N)} H_N(x) = \max_{q \in [0,\rho]} E(q; \rho, 0) \,\,\,\ \mathrm{a.s.} \nonumber
\end{align}
Our proof of Theorem \ref{thm:main} establishes that $\{ E(q; \rho, 0) : q \in [0,\rho]\}$ is maximized at $q = \rho^2$. Thus, observing that $W \stackrel{d}{=} (G + G^{T})/\sqrt{2N}$, where $G = (G_{ij} ) \in \mathbb{R}^{N \times N}$ is a matrix of i.i.d. $\mathcal{N}(0,1)$ random variables, we have, 
\begin{align}
\lim_{N \to \infty}  \frac{1}{N^{3/2}}  \max_{x \in \Sigma_N(\rho_N)} (x, Gx) = \frac{1}{\sqrt{2}} E (\rho^2; \rho,0)\,\,\,\ \mathrm{a.s.} \label{eq:submatrix} 
\end{align}
To each $N\rho_{N} \times N \rho_{N}$ principal submatrix of the matrix $G$, assign a score, which corresponds to the sum of its entries. The LHS of \eqref{eq:submatrix} represents the largest score, as we scan over all $N \rho_{N} \times N \rho_{N}$ principal submatrices of $G$. In \cite{bhamidi2012}, the authors derive upper and lower bounds on the maximum average value of $k\times k$ submatrices in an iid Gaussian matrix, for $k \leq \exp(o(\log N))$.  Our corollary extends the results of \cite{bhamidi2012} to the case of principal submatrices with size $N \rho_{N} \times N \rho_{N}$, and provides a tight first order characterization of the maximum. We note that in spin-glass terminology, the score of the largest $N\rho_{N} \times N\rho_{N}$ submatrix (not necessarily principal) corresponds to the ground state of a bipartite spin-glass model\cite{barra2010bipartite}, which is out of reach of current techniques.

Returning to the planted model \eqref{eq:model}, we complete our discussion by establishing that when the signal $\lambda$ is appropriately large, approximate support recovery is easily achieved using a spectral algorithm. To this end, we introduce the following two-step estimation algorithm, which rounds the leading eigenvector of $A$. 
 \begin{enumerate}
\item Let $\hat{v} = (\hat{v}_i)$ denote the eigenvector corresponding to the largest eigenvalue of the matrix $A$, with $\| \hat{v} \|^2 = N\rho$. Denote $\tilde{S} = \{ i \in [N] : \hat{v}_i \geq \frac{\delta}{2} \}$. 
\item If $|\tilde{S}| < \rho N$, sample $( \rho N - |\tilde{S}|)$ elements at random from $[N] \backslash \tilde{S}$ and augment to the set $\tilde{S}$ in order to construct $\hat{S}$. 

\item Otherwise, sample $\rho N$ elements uniformly at random from $\tilde{S}$, and denote this set as $\hat{S}$. 

\item Finally, construct $v_{\hat{S}} = \mathbf{1}_{\hat{S}}$, i.e., a vector with ones corresponding to the entries in $\hat{S}$, and zeros otherwise. 
\end{enumerate}
\noindent
Observe that the set $\hat{S}$ depends on $\delta$. We keep this dependence implicit for notational convenience.
Then we have the following lemma. 
\begin{lem}
\label{lemma:rounding}
For any $\varepsilon >0$ and $\lambda > (1+\varepsilon) \frac{1}{\rho}$, there exists $\delta := \delta(\varepsilon)$ such that with high probability as $N\to \infty$, $v_{\hat{S}}$ recovers the support approximately. 
\end{lem}

\subsection{Comparison with existing results} 
{In this section, we compare our approach with existing results, and summarize our main technical contributions. The information theoretic limits for exact recovery of planted submatrices under Gaussian additive noise was derived in \cite{butucea2015sharp}. \cite{cai2017submatrix} studies the submatrix localization problem under additive sub-Gaussian noise, and characterizes the information theoretic threshold for exact recovery. Further, they derive a computational lower bound via an average case reduction to the planted clique problem. Their results establish a crucial distinction between submatrix detection and localization problems---in contrast to the detection problem, there always exists a gap between the information theoretic threshold and the computational one in the localization problem. \cite{bala2011submatrix,kolar2011localization} study the performance of specific thresholding approaches for the submatrix localization problem. We emphasize that this line of work focuses on \emph{exact recovery}, and thus these results are not directly comparable with our results. }

{Closer in spirit to our work are the recent results in \cite{ding2019subexponential}---they study exact recovery for the planted Wigner model \eqref{eq:model}. Their results correspond to planted $k \times k$ principal submatrices with $k =o(n)$, and thus do not compare directly with our results. Based on the low degree likelihood ratio method, they suggest the existence of a hard regime if 
\begin{align}
\sqrt{\frac{k}{N}} \leq \lambda \rho \leq \min \Big\{ 1, \frac{k}{\sqrt{N}} \Big\}. \nonumber 
\end{align} 
Note that while their conjecture is restricted to the regime $k=o(N)$, their thresholds formally agree with our conclusions upon setting $k=N\rho$. In particular, it suggests that the threshold for the spectral algorithm derived in Lemma \ref{lemma:rounding} is the computational threshold for this problem. Improving our OGP result to the entire regime $\sqrt{\frac{1}{\rho} \log \frac{1}{\rho}} \ll \lambda \ll \frac{1}{\rho}$ is an intriguing open problem. We leave this for future study. Motivated by these conjectures,  \cite{arous2020free} have studied the regime $k=o(N)$ using the lens of OGP, and have provided} {rigorous structural evidence to the existence of a hard phase {for exact recovery.} }

The results of \cite{lelarge2019symmetric} derive the sharp (up to constant) threshold for approximate recovery in a related Bayesian problem. Assuming a product prior on the spike $v$ \eqref{eq:model},  they characterize the limiting (normalized) mutual information. We note that in this Bayesian setting, the analysis is substantially simplified by the  ``Nishimori Identity"---in statistical physics parlance, the problem is ``replica symmetric", and thus the limiting mutual information is expressed in terms of a simple scalar optimization problem. In contrast, the MLE corresponds to the ground state of the model, and in general is exhibits ``full replica symmetry breaking". To analyze its performance, we need to leverage recent advances in the study of mean field spin glasses, specifically related to the analysis of Parisi type formulas. We believe that our approach can also be useful in analyzing the landscape of the posterior distribution in high dimensional inference problems---we leave this for future study. Finally, the limiting (normalized) mutual information has been derived in a version of this problem, corresponding to the $k=o(N)$ setting \cite{barbier2020all}. These results have been derived by an appropriate extension of the \emph{adaptive interpolation method}. 

\subsection{Technical Contributions} 
Let us now discuss the main technical contributions of our paper. The proof of Theorem \ref{thm:main} is in two parts: we 
develop a variational formula for the large $N$ limit of the maximum likelihood constrained to a given overlap,  and analyze this variational problem.

We begin by viewing $\Sigma_N(\rho_N)$ as the configuration space of a spin system with energy $H_N$.
We compute the free energy of this spin system on the subspace of fixed overlaps. 
On this subspace, $H_N$ can be viewed as a two species generalized Sherrington-Kirkpatrick model,
and we compute the free energy via Guerra interpolation and the Aizenman--Sims--Star scheme \cite{ASS03}.
Due to the special symmetries of the problem considered here, 
we can develop a Parisi-type formula for the free energy as the sum of two Parisi type functionals and a Lagrange multiplier term
in the spirit of \cite{panchenkopspin}, with modifications to account for the change in alphabet, the additional constraints, and the ``multiple species''.
The key observation is that, due to these symmetries, we can decouple the two species 
through a Lagrange multiplier argument. 
With the positive temperature free energy in hand,  it remains to compute the zero temperature limit. We do so using the results 
of \cite{jagannath2017sen}, which computed the zero-temperature $\Gamma$-limit of the non-linear term in general Parisi-type functionals.

Let us pause to comment on the relation between this approach and prior work. 
Free energies of spin glasses with multiple species and non-symmetric alphabets have been studied in other works \cite{panchenko2015multispecies,panchenko2017potts,panchenko2017vectorspin,JKS2018}, 
and typically require substantially more sophisticated tools than what we use here. In particular, prior results crucially use Panchenko's synchronization mechanism.
We emphasize that the free energy we consider here cannot be directly obtained from these works. In principle, one might compute it by extending the work of \cite{JKS2018} to more general alphabets --- 
 we expect this would require significantly more advanced machinery from the theory of spin glasses and variational calculus. 
 Furthermore, the variational problems obtained by this approach are typically intractable. 
 Our approach bypasses these issue, and identifies a simple formula in this special case.

Let us now turn to the analysis of the ``ground state energy" functional. We establish the overlap gap property by analyzing the sign changes of the first derivative of the restricted ground state energy functional. This analysis requires a precise understanding of the scaling properties of the derivative as $\rho \to 0$. To this end, we utilize the natural connection of Parisi functionals with a family of SDEs, which, in turn, can be explicitly analyzed in the regime $\rho \to 0$.     

\noindent
\subsection*{Outline} 
The remainder of this paper is structured as follows. Theorem \ref{thm:small_signal_main} is established in Section \ref{sec:small_signal}. The first part follows by a data processing argument, and the second follows
by a Slepian-type bound. 
Theorem \ref{thm:main} is the main contribution of this paper, and is established in Section \ref{sec:ogp}. The proof depends crucially on Theorem \ref{thm:var-form-gs}, which derives  a Parisi type variational problem for the restricted energy $E(\cdot;  \rho, \lambda)$. In turn, to derive Theorem \ref{thm:var-form-gs}, we first establish a finite temperature Parisi formula (Proposition \ref{prop:generalized-multispecies}) in Section \ref{sec:free_energy}, and subsequently compute a limit of this formula as the temperature converges to zero (see Sections \ref{sec:ground_state} and \ref{sec:apriori-proof}). 
 Similar zero temperature formulae have been instrumental in establishing properties of mean field spin glasses in the low-temperature regime (see e.g. \cite{auffinger2017gs,auffinger2017sk,jagannath2017tobasco2, jagannath2017sen}). We emphasize that even with Theorem \ref{thm:var-form-gs}, the proof of Theorem \ref{thm:main} is subtle, and critically depends on understanding the scaling of the variational formula as $\rho \to 0$. Lemma \ref{lemma:rounding} is established in Section \ref{sec:rounding}. Finally, Section \ref{sec:free_energy_wells} studies the effect of OGP on this problem, and establishes that certain local Markov Chain based recovery algorithms are stymied by this structural barrier.

\vspace{.1in}
\noindent
\textbf{Acknowledgments.} The authors thank an anonymous referee for pointing out a substantial improvement to \prettyref{thm:small_signal_main} as well as for several constructive comments that have improved the exposition 
of this paper.  SS thanks Yash Deshpande for introducing him to the problem. DG gratefully acknowledges the support of  ONR grant N00014-17-1- 2790. AJ gratefully acknowledges  NSERC [RGPIN-2020-04597, DGECR-2020-00199] and the partial support of NSF grant NSF OISE-1604232.
Cette recherche a \'et\'e financ\'ee par le Conseil de recherches en sciences naturelles et en g\'enie du Canada (CRSNG).

\section{Proof of the overlap gap property}
\label{sec:ogp}

To establish Theorem \ref{thm:main}, we first require 
the following variational formula for the limiting constrained energy \eqref{eq:restricted}.
Let $\cM([0,q])$ denote the space of non-negative Radon
measures on $[0,\rho]$ equipped with the weak-{*} topology. Let $\mathcal{A}$
denote the set 
\[
\mathcal{A}=\left\{ \nu\in\cM([0,q]):\nu=m(s)ds+c\delta_{\rho},\:m(s)\geq0\text{ non-decreasing}\right\} 
\]
For $\nu\in\cA$ and $\mathbf{\Lambda}=(\Lambda_{1},\Lambda_{2})\in\R^{2}$,
let $u_{\nu}^{i}: [0, \rho] \times \mathbb{R} \to \mathbb{R}$ solve the Cauchy problem
\begin{equation}\label{eq:gspde}
\begin{cases}
\partial_{t}u_{\nu}^{i}+ 2(\partial_x^2 u_{\nu}^{i}+m(t)(\partial_{x}u_{\nu}^{i})^{2} )  =0 & (t,x)\in [0,\rho)\times \R \\
u_{\nu}^{i}(\rho,x)=(x+\Lambda_{i}+2c)_{+}, & x\in \R
\end{cases}
\end{equation}
 where $(\cdot)_{+}$ denotes the positive part, and $\partial_x^2$ is the Laplace operator. Note that any $\nu \in \cA$
 is locally absolutely continuous with respect to the Lebesgue measure on $(0,1)$, so that $m$ is almost surely well-defined. 
 As explained in \cite[Appendix A]{jagannath2017sen}, 
there is a unique weak solution to this PDE. 
Consider then the functional 
\[
\cP(\nu,\mathbf{\Lambda})=\rho u_{\nu}^{1}(0,0)+(1-\rho)u_{\nu}^{2}(0,0)-\Lambda_{1}q-\Lambda_{2}(\rho-q)- 2\int sd\nu(s).
\]
We then have the following,
\begin{thm}\label{thm:var-form-gs}
{For $q_N,\rho_N$ as above}, we have that 
\[
E(q; \rho, \lambda) := \lim_{N\to\infty}E_{N}(q_N;\rho_N,\lambda)= \lambda q^2+  \min_{\nu\in\cA,\Lambda\in\R^{2}}\mathcal{P}(\nu,\mathbf{\Lambda}),
\]
almost surely. 
\end{thm}
We defer the proof of this result to Section \ref{sec:ground_state}. Let us now complete the proof
of the overlap gap property, Theorem \ref{thm:main}, assuming Theorem \ref{thm:var-form-gs}.

\begin{proof}[Proof of Theorem \ref{thm:main}]
By \prettyref{thm:var-form-gs}, we have that 
\[
E(q;\rho,\lambda)=\lambda q^{2}+ \min_{\nu\in\cA,\mathbf{\Lambda}}\cP(\nu,\mathbf{\Lambda}).
\]
Observe that setting $c=\nu(\{1\})$, one may make the linear transformation
$\Lambda\mapsto\Lambda+2 c$ without changing the value of the functional.
Thus it suffices to consider the problem 
\begin{equation}
\min_{\nu\in\cA_{0},\Lambda\in\R^2}\cP(\nu,\Lambda)\label{eq:reduced-problem}
\end{equation}
where $\cA_{0}$ are those $\nu\in\cA$ with $\nu(\{\rho\})=c=0$.
By \prettyref{lem:strict-convexity},
we have that $\cP$ is jointly strictly convex in $(m,\Lambda)$ when restricted to this subspace. Thus
the minimum $(\nu_*,\Lambda_*)$ is unique. As such we can restrict this problem further to 
the compact set $S=\{\nu:\norm{\nu}_{TV}\leq \norm{\nu_*}\}\cup B_\delta(\Lambda_*)$ for some $\delta>0$. 
As $\lambda q^2+\cP(\nu,\Lambda)$ is convex in $q$, we may use  Danskin's envelope theorem, to show that 
\[
E(q;\rho,\lambda) = \lambda q^2+ \min_{S} \mathcal P(\nu,\Lambda)
\]
is differentiable in $q$ with derivative
\begin{equation}\label{eq:derivative}
\frac{\partial E}{\partial q}=2\lambda q+ (\Lambda_{2}^{*}-\Lambda_{1}^{*}),
\end{equation}
where $\Lambda_{i}^{*}$ denote the maximizers of \eqref{eq:reduced-problem}
corresponding to $q$. On the other hand one can show that $u_{\nu}^{i}$ is classically differentiable in its final time data,
and thus in $\Lambda_{i}$
for $t<\rho$ by a standard differentiable
dependence argument (see, e.g., \cite[Lemma A.5]{benarous2018jagannath} for a similar agument),. Thus we may differentiate
in $\Lambda$ to obtain the following fixed point equation for the
optimal $\Lambda^{*}:$
\begin{equation}\label{eq:fixed_point}
\begin{aligned}
\frac{\partial}{\partial\Lambda_{1}}u_{\nu}^{1}(0,0) & =\frac{q}{\rho}\\
\frac{\partial}{\partial\Lambda_{2}}u_{\nu}^{2}(0,0) & =\frac{\rho-q}{1-\rho}.
\end{aligned}
\end{equation}
Recall that $u^{i}_{\nu}$ solves the PDE \eqref{eq:gspde}, and thus $\partial_{\Lambda_i} u^{i}_{\nu}$ is given by
the solution $w^i = \partial_{\Lambda_i} u^{i}_{\nu}$ to 
\begin{equation}\label{eq:differentiated-pde}
\begin{cases}
\partial_t  w^i + L_i  w^{i} = 0 \\
w^i (\rho, x) = \indicator{x+\Lambda_{i}\geq0}, 
\end{cases}
\end{equation} 
where $L_i = 2 \partial_x^2 + 4 m(t) \partial_x u^{i}_{\nu}\partial_x$ is an elliptic operator.
Observe that $L_i$ is the infinitesimal generator of the diffusion $X^i_t$, given by
the solution to the stochastic differential equation
\begin{equation}\label{eq:zero-temp-diffusion}
dX_t^i = 2dB_t + 4m(t)\partial_x u^i_{\nu}(t,X^i_t)dt.
\end{equation}
By Ito's lemma we have, for $t \in (0,\rho)$, $x \in \mathbb{R}$, the stochastic representation
formula for $\partial_{\Lambda_i}u^i_{\nu}$, 
\begin{align*}
\partial_{\Lambda_i} u^{i}_{\nu} (t,x) = \prob(X_{\rho}^{i} + \Lambda_i \geq 0 | X^{i}_t = x). \nonumber 
\end{align*}
In particular, for $t = x = 0$, we have 
\begin{align}
\partial_{\Lambda_i} u^{i}_{\nu} (0, 0) = \prob(X_{\rho}^{i} + \Lambda_i \geq 0 | X^{i}_0 = 0). \label{eq:derivative1}
\end{align}
\noindent
Thus \eqref{eq:fixed_point} and \eqref{eq:derivative} identifies
$\Lambda_{i}^{*}$ with quantiles of $X_{\rho}^{i}$. 

Next, we derive bounds on
$\Lambda_1^*, \Lambda_2^*$
by analyzing the diffusion itself. 
Since $\partial_{x}u^{i}_{\nu}$ is weakly differentiable, we see that
it weakly solves \eqref{eq:differentiated-pde} as well.
Thus
\begin{equation}\label{eq:u-x-form}
\partial_{x}u^{i}_{\nu}(t,x)=\prob(X_{\rho}^{i}+\Lambda_{i}\geq0\vert X_{t}^{i}=x)
\end{equation}
as well. 
In particular we obtain the maximum principle $0\leq\partial_{x}u^{i}_{\nu}\leq1$.
Finally we require the following estimate on $\int_{0}^{\rho}m(t)dt$ for the unique minimizer $m(s)$.
\begin{lem}\label{lem:apriori_estimate}
We have that $\int_{0}^{\rho}m(t)dt\leq \frac{1}{2} \sqrt{\rho\log(1/\rho)}.$
\end{lem}
\noindent  We defer the proof of this estimate to \prettyref{sec:apriori-proof} and complete the proof assuming this estimate.  

Using $0\leq\partial_{x}u^{i}_{\nu}\leq1$ to control the drift term in \eqref{eq:zero-temp-diffusion}, 
we bound the above probability as 
\[
\prob\left( 2 B_{\rho}\geq-\Lambda_{i}\vert B_{0}=0\right)\leq\prob\left( X_{\rho}^{i}\geq-\Lambda_{i}\vert X_{0}^{i}=0\right)\leq\prob\left( 2B_{\rho}\geq-\Lambda_{i}-4\int_{0}^{\rho}m(t)dt\right).
\]
Combining this with the stationary point conditions for $\Lambda_{j}^{*}$ \eqref{eq:fixed_point},
and Lemma \ref{lem:apriori_estimate}, we have,  setting $\Phi$ as the 
 CDF of a standard Gaussian random variable,  
\begin{equation}\label{eq:bounds}
\begin{aligned}
2\sqrt{\rho}\,{\Phi}^{-1}\left(\frac{q}{\rho}\right)-2\sqrt{\rho\log\frac{1}{\rho}} & \leq\Lambda_{1}^{*}\leq 2\sqrt{\rho}\,{\Phi}^{-1}\left(\frac{q}{\rho}\right)\\
2\sqrt{\rho}\,{\Phi}^{-1}\left(\frac{\rho-q}{1-\rho}\right)-2\sqrt{\rho\log\frac{1}{\rho}} & \leq\Lambda_{2}^{*}\leq 2\sqrt{\rho}\,{\Phi}^{-1}\left(\frac{\rho-q}{1-\rho}\right).
\end{aligned}
\end{equation}
Armed with these bounds on the optimizers $\Lambda_1^*, \Lambda_2^*$, we can complete the proof in a relatively straightforward manner. First, observe that at $q = \rho^2$, \eqref{eq:fixed_point} ensures that $\Lambda_1^* = \Lambda_2^*$, and consequently, \eqref{eq:derivative} implies  
\begin{align}
\frac{\partial}{\partial q} E(\rho^2; \rho, \lambda) = 2 \lambda \rho^2 >0. \nonumber 
\end{align}

Finally, we evaluate $\partial_{q} E(q, \rho, \lambda)$ at $q = \rho^{2 - \delta}$, 
In particular, since $\alpha <  2 -\sqrt{2}$ and $\lambda \leq \rho^{-\alpha}$, choose $\delta$ such that 
\begin{align}
\alpha < \frac{3-2\delta}{2} < 2-\sqrt{2}. \nonumber 
\end{align} 

 In this case, we again have, using \eqref{eq:derivative}, 
\begin{align}
\partial_{q} E(\rho^{2 - \delta}; \rho, \lambda) = 2 \rho^{2- \delta} \lambda +  (\Lambda_2^* - \Lambda_1^*). 
\end{align}
Again, using the bounds derived in \eqref{eq:bounds}, we have, 
\begin{align*}
\partial_{q} E(\rho^{2 - \delta} ; \rho, \lambda) &\leq 
 2 \rho^{2- \delta} \lambda + 2 \sqrt{\rho}\,  \Phi^{-1} \Big( \frac{\rho - \rho^{2-\delta} }{1- \rho} \Big)
  - 2 \sqrt{\rho}\, \Phi^{-1}(\rho^{1-\delta}) + 2 \sqrt{\rho \log \frac{1}{\rho} },\\
&\leq 2 \lambda\rho^{2-\delta}-  2\sqrt{ 2\rho \log \frac{1}{\rho}} (1+ o_{\rho}(1))  + (2 \sqrt{ (1-\delta)}+ \sqrt{2}) \sqrt{2\rho \log \frac{1}{\rho} } ) \\
&\leq \sqrt{2\rho\log{\frac{1}{\rho}}}\left(\lambda\sqrt{\frac{\rho^{3-2\delta}}{\log(\frac{1}{\rho})}}-2+2\sqrt{(1-\delta)}+ \sqrt{2}+o(1)\right).
\end{align*}
If we take $\rho$ sufficiently small, we obtain
\[
\partial_q E(\rho^{2-\delta};\rho,\lambda) < 0
\]
by our choice of $\delta$.

The above calculation establishes that there exists and $\varepsilon' >0$ such that for any $C_1>0$, $\rho>0$ sufficiently small, and $\lambda \in \Big(C_1\sqrt{ \frac{\log \frac{1}{\rho} }{\rho}},\rho^{-\alpha} \Big)$, there exist $\rho^2< \tilde{q}_0 < \rho^{1+\varepsilon'}$ such that 
\begin{align*}
\partial_{q} E(\rho^2; \rho, \lambda) >0, \,\,\,\,\, \partial_{q} E(\tilde{q}_0; \rho, \lambda)<0.
\end{align*}
Thus there exist $w<\rho^2 < x < y = \tilde{q}_0$ such that 
\begin{align}
\max\{E(w;\rho,\lambda) , E(y; \rho, \lambda)\} < E(x; \rho,\lambda). \nonumber
\end{align}
Next, we claim that if we choose $C_1>2$,  
then $\max_{(w,\theta \rho)} E(q;\rho, \lambda) < \max_{(\theta \rho, \rho)} E(q; \rho, \lambda)$
for $0<\theta<\theta_0$ for some $\theta_0=\theta_0(C_1)>0$. Assuming this claim, if we take $\varepsilon=\varepsilon'\wedge\theta_0$,  we see that the points $w<x<y$ satisfy (i)-(iii) in definition \ref{defn:ogp}. In particular, $\varepsilon$-OGP holds as desired. 

As the proof of this claim is subsumed by our proof of Part 2 of Theorem~\ref{thm:small_signal_main} we end by proving said theorem. (Specifically, see \eqref{eq:energy-inequality} below and take $\theta_0(C_1) = c(C_1-2)$ there.)
\end{proof} 

\begin{proof}[\textbf{\emph{Proof of Part 2 of Theorem~\ref{thm:small_signal_main}}}]
Fix $\varepsilon>0$, and let $0<c=c(\varepsilon)<1$, 
to be specified later. Recall again that by Slepian's comparison inequality \cite{BLM13},  comparing $H_N$ to an IID process 
with the same variance, yields
\begin{equation}\label{eq:slepian}
\E \max_{x\in A} H_N(x) \leq \sqrt{2 \max_{x} \mathrm{Var}(H_N(x))\log \abs{A}},
\end{equation}
for any $A\subset\Sigma_N$.
Applying \eqref{eq:slepian} in the case $A=\Sigma_N(\rho_N)$, yields
\begin{equation}\label{eq:null-upperbound}
\lim_{N \to \infty} \frac{1}{N} \mathbb{E}\Big[ \max_{x \in \Sigma_N(\rho_N)} ( x, Wx ) \Big] \leq 2 \sqrt{ \rho^3 \log \frac{1}{\rho} }(1+o_\rho(1)). 
\end{equation}

By  \eqref{eq:slepian} 
and  the concentration of Gaussian maxima \cite[Theorem 5.8]{BLM13}, with high probability as $N \to \infty$, 
\begin{align}
\frac{1}{N} \max_{x \in \Sigma_N(\rho_N) : ( x, v ) \leq c N \rho}  \Big[ \frac{\lambda}{N} (x, v)^2 + H_N(x) \Big] 
&\leq \lambda c^2 \rho^2 + 2 \sqrt{\rho^3 \log \frac{1}{\rho}} (1+ o_{\rho}(1)) + o_N(1). \nonumber 
\end{align}
On the other hand, plugging-in $x=v$,
\begin{align}
\frac{1}{N} \max_{x \in \Sigma_N(\rho_N) : ( x, v) > c N \rho}  \Big[ \frac{\lambda}{N} (x, v)^2 + H_N(x) \Big] 
&\geq \lambda \rho^2 + o_N(1). \nonumber 
\end{align}
Thus selecting $c$ such that $(2+\varepsilon)\cdot(1- c^2)> 2$, 
we have 
\begin{align}
\frac{1}{N} \max_{x \in \Sigma_N(\rho_N) : ( x, v ) \leq c N \rho}  \Big[ \frac{\lambda}{N} (x, v)^2 + H_N(x) \Big] 
< \frac{1}{N} \max_{x \in \Sigma_N(\rho_N) : ( x, v ) > c N \rho}  \Big[ \frac{\lambda}{N} (x, v)^2 + H_N(x) \Big] \label{eq:energy-inequality} 
\end{align}
for all small enough $\rho>0$, implying $( \hat{x}_{\mathrm{MLE}}, v ) \geq c N \rho$ with high probability as $N \to \infty$. 
We conclude that the MLE recovers the support approximately in this regime. 
\end{proof}

\section{From Overlap Gap Property to Free Energy Wells}
\label{sec:free_energy_wells}

In this section, we will show that the overlap gap property implies a hardness
type result for Monte Carlo Markov chains. To this end, let us first
define the relevant dynamics. Fix $\beta>0$ and consider the Gibbs
distribution 
\[
\pi_{\beta}(dx)\propto e^{\beta (x,Ax)} dx,
\]
where $dx$ is the counting measure on $\Sigma_{N}(\rho)$. Construct a graph $G_N$ with vertices $\Sigma_N(\rho)$, and add an edge between $x,x' \in \Sigma_N(\rho)$ if and only if their Hamming distance is exactly 2. 
Let $(X_{t})_{t\geq0}$ denote any nearest neighbor Markov chain on $G_N$, reversible with respect to the stationary distribution $\pi_{\beta}$. 
By this we mean that the transition matrix $Q$ for $X_{t}$ satisfies detailed balance
with respect to $\pi_{\beta}$ and $Q(x,y)=0$ if $x$ and $y$ are not connected by an edge. 
We show here that when OGP holds, if we run the Markov Chain with 
initial data, $\pi_{\beta}(dx\vert \mathcal{I})$, where $\mathcal{I}$ is as in \prettyref{thm:main-MCMC} and $\beta$ is sufficiently large, 
it takes at least exponential time
for the chain to hit the region of order $\rho$ overlap. 

\subsection{Free energy wells}
Let $G=(V,E)$ be a finite graph and $\nu$
denote a probability measure on $V$. For $a\in\R$ and an $\epsilon>0$,
let $B_{\epsilon}(a)=[a-\epsilon,a+\epsilon]$ . For any function
$f:V\to\R,$ consider the following ``rate function'', 
\[
I_{f}(a;\epsilon)=-\log\nu(\{x:f(x)\in B_{\epsilon}(a)\}).
\]
For any two vertices $x,y\in V$ we say that $x\sim y$ if $x$ and
$y$ are connected by an edge.  We say that a function $f:V\to \R$ is
$K$-Lipschitz if there is some $K$ such that
\[
\max_{x\sim y} \abs{f(x)-f(y)} \leq K.
\]
\begin{defn}
We say that $f$ has an $\epsilon$-\emph{free energy well }of depth
$h$ in $[a,b]$ if there exists a $c\in[a,b]$ such that $B_{\epsilon}(a),B_{\epsilon}(b)$ and $B_{\epsilon}(c)$
are disjoint and
\[
\min\left\{ I_{f}(a,\epsilon),I_{f}(b,\epsilon)\right\} -I_{f}(c,\epsilon)\ge h.
\]
\end{defn}
\noindent
We then have the following whose proof is an adaptation
of  \cite[Theorem 7.4]{BGJ2018} to this setting. See also \cite{gamarnik2019planted}.

\begin{thm}
\label{thm:FEW}Let $X_{t}$ denote a nearest neighbor Markov chain
on $G$ which is reversible with respect to $\nu$. If $f$ is $\epsilon$-Lipschitz and has an
$\epsilon$-free energy well of depth $h$ in $[a,b]$, then for $\mathscr{A}=f^{-1}\left([a,b]\right),$
the exit time of $\mathscr{A}$, denoted $\tau_{\mathscr{A}^{c}}$,
satisfies 
\[
\int Q_{x}(\tau_{\mathscr{A}^c}\leq T)d\nu(x\vert \mathscr{A})\leq Te^{-h}
\]
for any $T$, where $Q_{x}$ is the law of $X_{t}$.
\end{thm}
\begin{proof}
In the following, let $\mathscr{A}_{-\epsilon}=f^{-1}\left([a,b]\backslash(B_{\epsilon}(a)\cup B_{\epsilon}(b))\right),$ $\mathscr{A}=f^{-1}([a,b]),$
and $B_{\epsilon}= \mathscr{A}\backslash \mathscr{A}_{-\epsilon}.$ Let us define the boundary
of $\mathscr{A}$ to be the set 
\[
\partial \mathscr{A}=\{x\in \mathscr{A}:\exists y\in \mathscr{A}^{c}:x\sim y\}.
\]
Observe that since $f$ is $\epsilon$-Lipschitz, $\partial \mathscr{A} \subset B_\epsilon$.

Let $\bar{X}_{t}$ be the Markov chain defined on $\mathscr{A}$
which is $X_{t}$ reflected at the boundary of $\mathscr{A}$. That
is, $\bar{X}_{t}$ has transition matrix $(\bar{Q}(x,y))$ which is
identical to $\mathscr{A}$ if $x\in \mathscr{A}\backslash\partial \mathscr{A}$
and $y\in \mathscr{A}$ and for $x\in\partial \mathscr{A}$, 
\[
\bar{Q}(x,y)\propto\begin{cases}
Q(x,y) & y\in \mathscr{A}\\
0 & \mathrm{else}.
\end{cases}
\]
Note that by detailed balance $\bar{X}_{t}$ is reversible with respect to $\tilde{\nu}=\nu(\cdot\vert \mathscr{A})$,
the invariant measure of $X_{t}$ conditioned on $\mathscr{A}$.
Let $\tau_{\partial \mathscr{A}}$ denote the first time either $X_{t}$ or $\bar{X}_{t}$ hits $\partial A$. Note that for $t\leq\tau_{\partial \mathscr{A}}$, the Markov Chains
$X_{t}$ and $\bar{X}_{t}$, started from a common state in $\mathscr{A}$ follow the same trajectory. 
As a result, 
\[
\int_{\mathscr{A}} Q_{x}(\tau_{\partial \mathscr{A}}<T)d\tilde{\nu}(x)=\int_{\mathscr{A}}\bar{Q}_{x}(\tau_{\partial \mathscr{A}}<T)d\tilde{\nu}(x).
\]
We now estimate the right hand side. Since $\partial  \mathscr{A} \subset B_\epsilon$, 
\[
\int_{\mathscr{A}}\bar{Q}_{x}(\tau_{\partial \mathscr{A}}<T)d\tilde{\nu}\leq\int_{\mathscr{A}}\bar{Q}_{x}(\exists i\leq T:X_{i}\in B_{\epsilon})d\tilde\nu\leq\sum_{i}\int_{\mathscr{A}}\bar{Q}_{x}(X_{i}\in B_{\epsilon})d\tilde{\nu}=T\tilde{\nu}(B_{\epsilon})\leq Te^{-h},
\]
where the last equality follows by stationarity and the last inequality
follows by assumption that $f$ has an $\epsilon$-free energy well of
height $h$.
\end{proof}

\subsection{From the Overlap Gap Property to Free energy wells at low temperature}

We now establish that if the overlap gap property holds, then the overlap $m(x)=\frac{1}{N}(x,v)$ has a free energy
well for $\beta>0$ sufficiently large. 

In the following, let 
\begin{equation}\label{eq:Sigma-constrained}
\Sigma_{N}(\rho,q)=\left\{ x\in\Sigma_{N}(\rho):\sum_{i=1}^{N\rho} x_{i}=Nq,\sum_{i=N\rho+1}^{N}x_i = N(\rho-q)\right\}.
\end{equation}
In defining the set $\Sigma_{N}(\rho,q)$, we implicitly use the distributional invariance of $A$ under row/column permutations to assume, without loss of generality, that $v$ is of the form 
\begin{equation}\label{eq:v-form}
v=(1,\ldots,1,0,\ldots,0), 
\end{equation}
where the first $N\rho_N$ entries are 1 and the remaining are $0$.
For $\lambda, \beta>0$,  let 
\begin{equation}
F_{N}(\lambda, \beta,\rho,q)=\frac{1}{N }\E\log\left(\sum_{x\in\Sigma_{N}(\rho,q)}\exp\Big(\beta \big(\lambda N q^2+ H_{N}(x) \big) \Big)\right), \label{eq:finitefree}
\end{equation}
where $H_N$ is defined by \eqref{eq:hamiltonian} and $\Sigma_N(\rho,q)$ is defined in \eqref{eq:Sigma-constrained}. 
Let $F_N(\lambda,\beta,\rho_N)$ be defined similarly except with the sum running over the set $\Sigma_N(\rho_N)$.
Using \cite[Lemma 2.6]{jagannath2017sen}, we have,
\begin{equation}
\abs{ \frac{1}{\beta}F_{N}(\lambda, \beta,\rho_N,q_N)-\E\Big[ \frac{1}{N}\max_{x\in\Sigma_{N}(\rho_N,q_N)} (x,Ax) \Big]}\leq\frac{\log\abs{\Sigma_{N}(\rho_N,q_N)}}{N\beta}\leq\frac{C(\rho,q)}{\beta},\label{eq:trivial-bound}
\end{equation}
 for some $C(\rho,q)>0$ independent of $N$. Observe that by combining this bound with  \eqref{eq:restricted} implies that if the limit $F(\lambda, \beta,\rho,q)=\lim F_N(\lambda,\beta,\rho_N,q_N)$ exists, then 
\[
E(q;\rho,\lambda)=\lim_{\beta\to\infty} \frac{1}{\beta} F(\lambda, \beta,\rho,q). 
\]

\begin{thm}
\label{thm:GOP-to-FEW}
{For any $\varepsilon,\lambda,\rho>0$, if the $\varepsilon$-overlap gap property holds, 
then there are points $w<\rho^2<x<y<\rho$ with $y=o_\rho(\rho)$ such that for $N$ sufficiently large and any $\delta>0$,  the overlap $m(x)$ has an $\delta$-free energy well
of depth $N h$ in $[w,y]$ with probability $1-O(e^{-cN})$.}
\end{thm}
\begin{proof}
By definition of the $\varepsilon$-overlap gap property, we may take $w<\rho^{2}<x<y=o_\rho(\rho),$ such that 
\[
\max\{E(w;\rho,\lambda),E(y;\rho,\lambda)\}<E(x;\rho,\lambda).
\]
Furthermore, by continuity of the map $q\mapsto E(q;\rho,\lambda)$ (see \eqref{eq:derivative}),
we may assume that there is an $\delta>0$ and $h>0$, such that
\[
\max_{q\in B_{\delta}(w)\cup B_{\delta}(y)}E(q;\rho,\lambda)-\min_{q'\in B_{\delta}(x)}E(q';\rho,\lambda)\leq-h
\]
and such that $B_{\delta}(w),B_{\delta}(y)$, and $B_{\delta}(x)$
are pairwise disjoint. By \eqref{eq:trivial-bound}, we then have
that 
\[
\max_{q\in B_{\delta}(w)\cup B_{\delta}(y)}F_{N}(\lambda, \beta, q)-\min_{q'\in B_{\delta}(x)}F_{N}(\lambda, \beta, q')\leq-\frac{h}{2}
\]
 for $\beta$ sufficiently large. 

For a set $E\subset\Sigma_N$, let $Z_{N}(E)=\sum_{x\in E}e^{\beta (x,Ax)}$, and let $\cA_{N}=m(\Sigma_N(\rho))$ denote the image of the overlap function. 
Note that $|\cA_{N}\cap B_{\delta}(a)|\leq C\cdot N\cdot\delta$ for some constant $C>0$. 
By a union bound, we have 
\[
\abs{\frac{1}{N}\log Z_{N}(B_{\delta}(a))-\max_{q\in B_{\delta}(a)\cap\cA_N}\frac{1}{N}\log Z_{N}(\Sigma_{N}(\rho,q))}\leq C\cdot\delta.
\]
Recall that by Gaussian concentration of measure \cite[Theorem 1.2]{panchenkobook}, there is a $C>0$
such that for $N\geq1$ and $\eta>0$, 
\begin{align*}
\max_{q \in \mathcal{A}_N}\frac{1}{N}\log \mathbb{P} \Big(\abs{F_{N}(\lambda, \beta,q)-\frac{1}{N}\log\left(Z_N(\Sigma_{N}(\rho,q)\right)}\geq\eta \Big) & \leq-C\eta^{2}\\
\frac{1}{N}\log \mathbb{P} \Big(\abs{F_{N}(\lambda, \beta)-\frac{1}{N}\log Z_N(\Sigma_{N}(\rho))}\geq\eta \Big) & \leq-C\eta^{2}.
\end{align*} 
If we take $\delta$ and $\eta$ sufficiently small, we see that
\[
\max \Big\{\frac{1}{N}\log Z_N(B_{\delta}(w)),\frac{1}{N}\log Z_N(B_{\delta}(y)) \Big\}-\frac{1}{N}\log Z_N(B_{\delta}(x))\leq-\frac{h}{4},
\]
 with probability $1-O(e^{-cN})$, where we have combined the above
concentration bounds with a union bound, using the fact again that
$\abs{\cA_{N}}\leq C\cdot N$. Setting $I_m$ as the rate-function corresponding to the overlap $m$ with respect to the measure $\pi_{\beta} \propto \exp(\beta (x,Ax))$ on $\Sigma_N(\rho)$, and subtracting the above from $\frac{1}{N}\log Z_{N}(\Sigma_{N}(\rho))$, we have,  
\[
\min\{I_{m}(a;\delta),I_{m}(b;\delta)\}-I_{m}(q_{1};\delta)\geq N\frac{h}{4}
\]
with probability $1-O(e^{-cN}).$ This yields the desired result. 
\end{proof}

\noindent
Observe that $m(x)$ is $1/N$-Lipschitz. Combining \prettyref{thm:FEW} with
\prettyref{thm:GOP-to-FEW} then immediately yields the following
corollary. 
\begin{cor}\label{cor:OGP-to-exit}
For any $\varepsilon,\lambda,\rho>0$, if the $\varepsilon$-overlap gap property holds, 
then there are points $w<\rho^2<x<y<\rho$ with $y=o_\rho(\rho)$ 
and an $h>0$ such that with probability $1-O(e^{-c'N})$, 
if $\mathcal{I}=(a,b)$, then the exit time of $\mathcal{I}$, $\tau_{\mathcal{I}^{c}}$,
satisfies
\[
\int Q_{x}(\tau_{\mathcal{I}^{c}}\leq T)d\pi(x\vert \mathcal{I})\leq Te^{-chN}
\]
for some $c>0$.
\end{cor}

\section{Variational formula for constrained energy}
\label{sec:ground_state}
We establish Theorem \ref{thm:var-form-gs} in this section. 
To begin we introduce a relaxation of the optimization problem \eqref{eq:ml_defn},
called the ``positive temperature free energy'' of the problem. 
{Recall that} we may assume, without loss of generality, 
that  $v$ is of the form
\begin{equation}
v=(1,\ldots,1,0,\ldots0), 
\end{equation}
where the first $N\rho_N$ entries are 1 and the remaining are $0$. Recall that $\rho_N$ is a sequence such that $N\rho_N \in \mathbb{N}$ and $\rho_N\to\rho$. 
For any $q \in [0, \rho]$, fix a sequence $q_N \to q$ such that $Nq_N\in \mathbb{N}$ and $Nq \in [Nq_N - 1/2, Nq_N + 1/2)$.   In the subsequent, we will refer to a sequence $(\rho_N,q_N)\to (\rho,q)$ that satisfies these conditions as an \emph{admissible sequence}.
We begin by deriving the following formula for the limiting free energy, $F(\beta,\rho, q)=\lim F_N(0,\beta,\rho_N, q_N)$, 
where $F_N$ is as in \eqref{eq:finitefree} and $\beta>0$ is fixed. 

To this end, let $\Lambda_{1},\Lambda_{2}\in\R$, and $\mu\in\cM_1([0,q])$, where $\cM_1([0,q])$ is the space 
of probability measures on $[0, q]$ equipped with the weak-* topology. Let
$u^{i}_{\mu,\beta}$ be the unique weak solutions to the Cauchy problem
\begin{align}
\begin{cases}
\partial_{t}u_{\mu,\beta}^{i}+ 2(\partial_x^2 u_{\mu,\beta}^{i}+\beta\mu([0,t])(\partial_{x}u_{\mu,\beta}^{i})^{2} )=0 & (t,x)\in[0,\rho)\times\R\\
u_{\mu,\beta}^{i}(\rho,x)=\frac{1}{\beta}\log\left(1+\exp(\beta(x+\Lambda_{i}))\right).
\end{cases}
\label{eq:GSPDE}
\end{align}
For a definition of weak solution as well as well-posedness of the Cauchy problem see \cite[Sec. 2]{jagannath2016tobasco} (alternatively, see  \cite[Appendix A]{benarous2018jagannath}).
Consider the functional 
\begin{align}
\cP_{\beta}(\mu,\Lambda_{1},\Lambda_{2};q) & =\frac{\log2}{\beta}-\Lambda_{1}q-\Lambda_{2}(\rho-q)+\rho u_{\mu,\beta}^{1}(0,0)+(1-\rho)u_{\mu,\beta}^{2}(0,0) \nonumber\\
 & \qquad-2 \int_0^\rho s\beta\mu([0,s])ds. \label{eq:functional}
\end{align}
We then have the following.   
\begin{prop}
\label{prop:generalized-multispecies} For $\beta>0$ and  any admissible sequence $(\rho_N,q_N)\to(\rho,q)$ we have that
$F(\beta,\rho, q)=\lim_{N\to\infty} F_{N}(0,\beta,\rho_N,q_N)$ exists and satisfies
\begin{equation}\label{eq:finite-beta}
\frac{1}{\beta}F (\beta,\rho,q )=\min_{\substack{\mu\in \cM_1([0,q]),\\
\Lambda_{1},\Lambda_{2}\in\R
}
}\cP_{\beta}(\mu,\Lambda_{1},\Lambda_{2};q).
\end{equation}
In particular, this minimum is achieved.
\end{prop}
\noindent
We defer the proof of this result to Section \ref{sec:free_energy}.

\subsection{Proof of variational formula}

We now compute the zero-temperature limit of the positive temperature problem. 
\begin{thm}
\label{thm:constrained-gs-formula}We have that 
\[
\lim_{\beta\to\infty} \frac{1}{\beta} F(\beta,\rho,q)=  \, \min_{\nu\in\cA,\mathbf{\Lambda}\in\R^2}\mathcal{P}(\nu,\mathbf{\Lambda};q)
\]
\end{thm}
\noindent Recall that by \eqref{eq:trivial-bound}, this immediately implies Theorem \ref{thm:var-form-gs}.

To this end, we study the convergence of the above variational problem as $\beta \to \infty$.
Let us first recall the notion of sequential $\Gamma$-convergence.
\begin{defn}
Let $X$ be a topological space. We say that a sequence of functionals $F_n : X \to [- \infty, \infty]$ \textit{sequentially $\Gamma$-converges} to $F : X \to [-\infty, \infty]$ if 
\begin{enumerate}
\item The $\Gamma-\liminf$ inequality holds: For every $x$ and sequence $x_n \to x$, $$\liminf_{n \to \infty} F_n(x_n) \geq F(x).$$ 
\item The $\Gamma-\limsup$ inequality holds: For every $x$, there exists a sequence $x_n \to x$ such that $$\limsup_{n \to \infty} F_n(x_n) \leq F(x).$$ 
\end{enumerate}
For a sequence of functionals $F_{\beta}$ indexed by a real parameter $\beta$, we say that $F_{\beta}$ sequentially $\Gamma$-converges to $F$ if for any sequence $\beta_n \to \infty$, the sequence $F_{\beta_n}$ sequentially $\Gamma$ converges to $F$. 
\end{defn}
 Recall that in  \cite[Theorem 3.2]{jagannath2017sen} it was shown that the 
 functionals 
\[
\cF_{\beta}(\nu,\Lambda_{i})=\begin{cases}
u_{\mu,\beta}^{i}(0,0) & \nu=\beta\mu([0,s])ds,\quad\mu\in\cM_1([0,q])\\
+\infty & \textrm{o.w.}
\end{cases}
\]
sequentially $\Gamma$-converges to the solution of \eqref{eq:gspde},
\begin{equation}\label{eq:old-gamma-result}
\cF_{\beta}(\nu,\Lambda_{i})\stackrel{\Gamma}{\to}u_{\nu}^{i}(0,0).
\end{equation}
Let
\begin{align*}
\cG_{\beta}(\nu,\Lambda_{1},\Lambda_{2}) & =\rho\cF_{\beta}(\nu,\Lambda_{1})+(1-\rho)\cF_{\beta}(\nu,\Lambda_{2})\\
\cG(\nu,\mathbf{\Lambda}) & =\rho u_{\nu}^{1}(0,0)+(1-\rho)u_{\nu}^{2}(0,0).
\end{align*}
Furthermore, let 
\begin{align*}
\cE_{\beta}(\nu,\mathbf{\Lambda}) & =\cG_{\beta}(\nu,\Lambda_{1},\Lambda_{2})-\Lambda_{1}q-\Lambda_{2}(\rho-q)- \int sd\nu(s),\\
\cE(\nu,\mathbf{\Lambda}) & =\cG(\nu,\mathbf{\Lambda})-\Lambda_{1}q-\Lambda_{2}(\rho-q)- \int sd\nu(s).
\end{align*}
The preceding results (with minor modification) will yield the following result.
\begin{lem}
\label{lem:gamma-conv} We have that $\cE_{\beta}\stackrel{\Gamma}{\to}\cE.$ 
\end{lem}

\begin{proof}
By \eqref{eq:old-gamma-result},we have the $\Gamma-\liminf$
inequality: for any sequence $(\nu_{\beta},\mathbf{\Lambda}^{\beta})$,
\[
\liminf\cG_{\beta}(\nu_{\beta},\Lambda_{1}^{\beta},\Lambda_{2}^{\beta})\geq\rho\liminf\cF_{\beta}(\nu_{\beta},\Lambda_{i}^{\beta})+(1-\rho)\liminf\cF_{\beta}(\nu_{\beta},\Lambda_{i}^{\beta})=\cG(\nu,\mathbf{\Lambda}).
\]
It remains to prove the $\Gamma-\limsup$ upper bound.  This will follow since the recovery sequence in the $\Gamma-\limsup$ upperbound of \eqref{eq:old-gamma-result} does not depend on the functional itself.

More precisely, fix $(\nu,\mathbf{\Lambda)}$. Consider the sequence  $(\nu_{\beta},\mathbf{\Lambda}^{\beta})$ with $\mathbf{\Lambda}^{\beta}=\mathbf{\Lambda}$,  and $\nu_\beta$ constructed as in \cite[Lemma 2.1.2]{jagannath2017tobasco2} (see alternatively \cite[Lemma 3.4]{jagannath2017sen}).  
Consequently, we have that 
\[
\lim\cG_{\beta}(\nu_{\beta},\mathbf{\Lambda})=\lim\cF_{\beta}(\nu_{\beta},\Lambda_{1})+\cF_{\beta}(\nu_{\beta},\Lambda_{2})=\cG(\nu,\mathbf{\Lambda}),
\]
along this sequence as each of the summands converge by applying \cite[Lemma 3.3]{jagannath2017sen} termwise.
\end{proof}
\begin{lem}
\label{lem:FTGC}Any sequence of minimizers of $\cE_{\beta}$ is pre-compact.
Furthermore, any limit point of such a sequence is a minimizer of
$\cE$ and 
\[
\lim\min\cE_{\beta}=\min\cE.
\]
\end{lem}

\noindent
In fact, this sequence is unique, however we will not require this. The compactness of $\nu_{\beta}$ is established in the following lemma, whose proof is deferred to \prettyref{sec:apriori-proof}. 
\begin{lem}\label{lem:apriori-estimate-2}
For every $\beta>0$,
\[
\| \nu_{\beta} \|_{TV}\leq \frac{1}{2} \sqrt{\rho\log(1/\rho)}.
\]
\end{lem}
\begin{proof}[\textbf{\emph{Proof of Lemma \ref{lem:FTGC}}}]
The compactness of $\nu_{\beta}$ follows from Lemma \ref{lem:apriori-estimate-2}. 
On the other hand, by \prettyref{lem:minimum} below,
 $\Lambda_{i}^{\beta}$ lie in a uniformly bounded
set. Thus the sequence is pre-compact. The second half of the result
follows by the fundamental theorem of $\Gamma$-convergence. 
\end{proof}

\begin{proof}[\textbf{\emph{Proof of \prettyref{thm:constrained-gs-formula}}}]
This follows by combining \prettyref{lem:gamma-conv}-\ref{lem:FTGC}, and
\prettyref{prop:generalized-multispecies}.
\end{proof}
\begin{proof}[\textbf{\emph{Proof of  Theorem \ref{thm:var-form-gs}. }}]
This follows immediately upon combining \eqref{eq:trivial-bound} with \prettyref{thm:constrained-gs-formula}. 
\end{proof}

\section{Bounds on optimal measures}\label{sec:apriori-proof}
In this section, we prove Lemmas \ref{lem:apriori_estimate} and \ref{lem:apriori-estimate-2}. 
 To this end, we need the following useful notation. Let $\bLambda =(\Lambda_1,\Lambda_2)$, and let us
 abuse notation to denote by $d\rho$, the two atomic measure on $\{1,2\}$,
\[
d\rho = \rho\delta_1 + (1-\rho)\delta_2.
\]
Let $v_i(t,x)$ be given by the change of variables
\begin{equation}\label{eq:v-def-2}
v_i(t,x) = \beta u_{\mu,\beta}^i(t,x/\beta;\Lambda_i/\beta).
\end{equation} 
Note that $v_i$ solves
\begin{equation}\label{eq:v-def}
\begin{cases}
\partial_t v + {2\beta^2}(\partial_x^2 v + \mu([0,s])(\partial_x v)^2 )= 0 &  (t,x)\in [0,\rho)\times\R^2\\
v(\rho,x) = \log(1+\exp(x+\Lambda)),
\end{cases}
\end{equation}
with $\Lambda=\Lambda_i$.
It is then helpful to rewrite the functional $P_\beta = \beta \cP_\beta$ in the following form, 
\[
P_\beta(\mu,\Lambda_1,\Lambda_2; q) = -q \Lambda_1 -(\rho-q)\Lambda_2 + \int v_i(0,0)d\rho(i) - 2\beta^2\int s\mu(s)ds.
\]
In the following, it will be useful to note the first order optimality conditions for this functional. 
To this end, let $\hat X_s$ solve the stochastic differential equation
\begin{equation}
d\hat{X}_s^{i} = 4 \beta^2 \mu([0,t]) \partial_x v_i(s, \hat{X}_s^{i}) +  2 \beta \,  dB_t \label{eq:opt_trajectory}
\end{equation}
with initial data $\hat{X}_0^{i}=0$, and let 
\[
G_{\mu,\bLambda}(t) = \int_t^\rho 4\beta^2 \Big(\int \E(\partial_x v_i)^2(s, \hat{X}^i_s)d\rho(i)-s \Big)ds.
\]

\begin{lem}
\label{lem:minimum}
We have the following.
\begin{enumerate}
\item For every $\beta>0$ there is a unique minimizing triple $(\mu,\bLambda)$ of $P_\beta$. 
Furthermore, the set of optimizing $\{\bLambda_\beta\}_\beta$ lie in a compact subset of $\R^2$.
\item This triple satisfies
the optimality conditions
\begin{align}
\mu(G_{\mu,\bLambda}(s) = \min_s G_{\mu,\bLambda}(s))&=1\label{eq:mu-min}\\
\int \partial_{\Lambda_i}v_i(0,0) d\rho(i) &= \rho.\label{eq:v-lambda}
\end{align}
\item In particular, for any $q\in\supp(\mu)$,
\begin{equation}
q = \int \E (\partial_x v_i)^2(q,\hat{X}^i_q)d\rho(i).\label{eq:fixed-point}
\end{equation}
\end{enumerate}
\end{lem}
\begin{proof}
We begin with item $(1)$, and first prove the uniqueness of the minimizing pair. To see this, we note 
that by \cite[Lemma 4.2]{jagannath2017sen},
the map
\[
(\mu,\Lambda)\mapsto w(0,0)
\]
where $w$ weakly solves \eqref{eq:v-def}
is strictly convex. Thus $P_\beta$ is strictly convex. 

To show existence of a minimizing pair, note that since $\cM_1([0,q])$ is weak-* compact, 
it suffices to show that $\bLambda$
lives in a compact subset of $\R^2$. By
the parabolic comparison principle (see \cite[Lemma 4.6]{jagannath2017sen} in this case),
it follows that, for any $\mu$,
\[
v_i(0,0) \geq \log(1+\exp(\Lambda_i)) .
\]
Thus
\[
P_\beta(\mu,\Lambda_1,\Lambda_2; q) \geq \log(1+\exp(\Lambda_1)) - q\Lambda_1+\log(1+\exp(\Lambda_2))-(\rho-q)\Lambda_2 -\int s ds,
\]
which diverges to infinity as $\Lambda_1,\Lambda_2\to\pm \infty$, from which the compactness result follows. In fact, this shows that 
the set is $\beta$ independent.

It remains to derive the optimality conditions. For item $(3)$, note that the fixed point equation \eqref{eq:fixed-point} for the support follows upon differentiating $G$ and applying \eqref{eq:mu-min}.  
To obtain \eqref{eq:v-lambda}, first note that
 $v_i$ are differentiable in $\Lambda_i$ --- this follows by a classical differentiable dependence argument, see, e.g., \cite[Lemma A.5]{benarous2018jagannath}.
Explicitly differentiating the functional in $\Lambda_i$, \eqref{eq:v-lambda} then follows upon observing the relation
\[
\rho \Big(\frac{q}{\rho} \Big)+(1-\rho) \Big(\frac{\rho-q}{1-\rho} \Big)=\rho.
\]
The first-order stationary condition for $G$ \eqref{eq:mu-min} then follows by first fixing $\Lambda_i$ and computing the first variation of the maps
\[
\mu \mapsto v_i(0,0;\Lambda_i).
\]
This has been done in \cite[Lemma 4.3]{jagannath2017sen} following \cite[Lemma 3.2.1]{jagannath2017tobasco}.  
In particular, this yields the following first variation formula for $P_{\beta}$: if $\mu_t$ 
is a weak-* right differentiable path in $\cM_1([0,q])$ ending at $\mu$, in the sense that 
\[
\dot \mu = \lim_{t\to0^+} \frac{\mu_t - \mu}{t}
\]
exists weak-*, then
\[
\frac{d}{dt}\vert_{t=0}P_\beta(\mu_t,\bLambda; q) = \int G_{\mu,\bLambda}(s)d\dot\mu.
\]
By the first order optimality condition for convex functions,
it follows that the righthand side is non-negative for all such paths if and only if we choose $\mu_0$ to be the optimizer of $P_{\beta} (\cdot,\bLambda; \textcolor{red}{q})$.
If we then take the path $\mu_t = t\delta_{s_0} + (1-t)\mu$,  we see that 
\[
G_{\mu,\bLambda}(s_0)\geq \int G_{\mu,\bLambda} d\mu,
\]
from which \eqref{eq:mu-min} follows. 
\end{proof}
\noindent
Recall the function $F$ introduced in Proposition \ref{prop:generalized-multispecies}. 
Armed with Lemma \ref{lem:minimum}, we next establish a formula for the $\beta$-derivative of $F$.
\begin{lem}\label{lem:diff-lemma}
We have that 
\[
\partial_{\beta}F( \beta, \rho,q )= 2\beta\int  (\rho^{2}-q^{2} ) d\mu(q).
\]
\end{lem}
\begin{proof}
We start with \eqref{eq:functional} and \eqref{eq:finite-beta},  and observe that we can equivalently express
\[
F(\beta, \rho,q )= \log 2 +  \min_{\Lambda_{i},\mu} P_{\beta}(\mu, \Lambda_1, \Lambda_2; {q}).
\]
By the same argument as  \cite[Theorem 4.1]{jagannath2017sen}, we see that 
\begin{align}
\partial_{\beta}v_{i}(0,0)=\beta\left\{ 4\rho \E\left(\partial_{x}^{2}v_{i}(\rho, \hat{X}^{i}_{\rho})+(\partial_{x}v_{i})^{2}(\rho, \hat{X}^{i}_{\rho})\right)- 4\int_{0}^{\rho}s\E(\partial_{x}v_{i})^{2}(s, \hat{X}^{i}_{s})d\mu\right\}, \label{eq:int2}
\end{align}
where $\hat{X}^{i}$ solves the stochastic differential equation \eqref{eq:opt_trajectory}.  \cite[Lemma A.5]{benarous2018jagannath} implies that  $w_i=\partial_{\Lambda_i} v(t,x)$ weakly solves
\[
\begin{cases}
\partial_t w_i +L_i w = 0 & (t,x) \in [0,\rho)\times\R \\
w_i(T,x) = \partial_{\Lambda_i} \log(1+\exp(x+\Lambda_i)) & x\in \R,
\end{cases}
\]
where $L_i$ is the infinitesimal generator for $\hat{X}^i$.  
Therefore, 
\[
\E\Big[ \partial_{\Lambda_{i}}v_{i}(\rho,\hat{X}^i_{\rho}) | \hat{X}_0^{i} =0  \Big]=\partial_{\Lambda_{i}}v_{i}(0,0).
\]
By direct computation and \eqref{eq:v-def},
\begin{align*}
\partial_x^2 v_i (\rho, x) + (\partial_x v_i)^2 (\rho,x) = \partial_{\Lambda_i} v_i (\rho, x). 
\end{align*}
Combining these observations with \eqref{eq:int2}, and using the optimality conditions \eqref{eq:v-lambda} and \eqref{eq:fixed-point} yields 
\begin{align*}
\partial_{\beta}F( \beta, \rho,q ) & =\rho\partial_{\beta}v_{1}(0,0)+(1-\rho)\partial_{\beta}v_{2}(0,0)- 4 \beta\int_{0}^{\rho}s\mu([0,s])ds\\
 & =\beta\left[4\rho\cdot \rho- 4\int_{0}^{\rho}s^{2}d\mu(s) \right]- 2\beta\int_{0}^{\rho}(\rho^{2}-s^{2})d\mu(s) \\
 & = 2 \beta \int(\rho^{2}-s^{2})d\mu(s)
\end{align*}
as desired.
\end{proof}
\noindent
We now turn to the proof of Lemmas \ref{lem:apriori_estimate} and \ref{lem:apriori-estimate-2}.

\begin{proof}[\textbf{\emph{Proofs of Lemmas \ref{lem:apriori_estimate} and \ref{lem:apriori-estimate-2}}}]
Let us begin by observing that Lemma \ref{lem:apriori_estimate} follows from Lemma \ref{lem:apriori-estimate-2}.
To see this, recall the functional $\cP_\beta$ from \eqref{eq:finite-beta} and let $\mu_\beta$ denote the 
corresponding minimizer. Then, by \prettyref{lem:FTGC},
there is a limit point of the sequence $\nu_\beta = \beta\mu_\beta([0,s])dt$, call it $\nu = \tilde m(t)dt+c\delta_\rho$,
and any such limit point is a minimizer of $\cP$. Observe that, by the reduction from \eqref{eq:reduced-problem},
and the strict convexity of the corresponding problem, $\tilde m = m$. 
We now observe that along any subsequence converging to $\nu$,
\begin{align}
\int_0^{\rho} m(t) \mathrm{d}t \leq \lim_{\beta \to \infty} \beta \int_0^{\rho} \mu_{\beta} ([0,t]) \mathrm{d}t. \label{eq:int3}
\end{align}
The desired bound then follows from Lemma \ref{lem:apriori-estimate-2}.

Let us now turn to the proof of Lemma \ref{lem:apriori-estimate-2}.
By Fubini's theorem, we have that
\[
\beta \int_0^{\rho } \mu_{\beta} ([0,t]) \mathrm{d}t = \beta \int_0^{\rho} (\rho - q) \mathrm{d}\mu_{\beta}(q) 
\leq \frac{1}{2 \rho} \beta \int_0^{\rho} (\rho^2 - q^2) \mathrm{d} \mu_\beta(q) = \frac{1}{4 \rho} \partial_{\beta} F( \beta,\rho, q), 
\]
where the last inequality follows using Lemma \ref{lem:diff-lemma}. 
Next, recall $F_N(\beta,\rho_N, q_N)$ from \eqref{eq:finitefree}. By differentiation, observe that $F_N(\beta,\rho_N,q_N)$ is convex in $\beta$, and thus Proposition \ref{prop:generalized-multispecies} implies that $F(\beta,\rho,q)$ is convex as well. Thus 
$\partial_\beta F_N(\beta,\rho_N, q_N) \to \partial_\beta F(\beta,\rho, q)$, by Griffith's lemma for convex functions. 
Consequently, we have, 
\begin{align}
\beta \int_0^{\rho} \mu_{\beta} ([0,t]) \mathrm{d}t \leq \frac{1}{4 \rho}  \lim_{N \to \infty} \partial_{\beta}F_N( \beta,\rho_N, q_N ). \nonumber
\end{align}
Finally, note that 
\begin{align}
\partial_{\beta} F_N({\beta},\rho_N, q_N ) = \frac{ 1}{N} \mathbb{E}[ \langle  H_N \rangle ] \leq \frac{1}{N}\E \max_{\Sigma_N(\rho_N,q_N)} H_N\leq  2 \sqrt{ \rho^3 \log \frac{1}{\rho} } (1+o_N(1)), \nonumber
\end{align}
where $\gibbs{\cdot}$ denotes integration with respect to the Gibbs measure $\pi(\{\sigma\}) \propto \exp(-H_N(\sigma))\indicator{\Sigma_N(\rho_N,q_N)}$,
the first equality follows by Gaussian integration-by-parts \cite[Lemma 1.1]{panchenkobook}, and
the last bound follows by  bounding $\langle H_N\rangle$ by the maximum  and applying 
Slepian's comparison inequality \eqref{eq:slepian} with $A = \Sigma_N(\rho_N,q_N)$. Thus we obtain, 
\begin{align*}
\int_0^\rho \beta \mu_\beta([0,t]) d t \leq \frac{1}{2} \sqrt{\rho \log 1/\rho},
\end{align*}
as desired. 
\end{proof}

\
\section{Proof of Free energy formula}
\label{sec:free_energy}

In this section, we aim to prove that { for $(\rho_N,q_N)\to(\rho,q)$ admissible,}
\begin{equation}\label{eq:free-energy-goal}
\lim_{N\to\infty} F_{N}(\beta,\rho_N,q_N)=\inf_{\substack{\mu\in\cM_1([[0,\rho])\\\Lambda\in\R^2}}P_{\beta}(\mu,\Lambda;q).
\end{equation}
Note that this agrees with the statement of \prettyref{prop:generalized-multispecies} after dividing through by $\beta$,
after noting that by Lemma \ref{lem:minimum}, the infimum is actually achieved. 

To this end, first write $v$ as in \eqref{eq:v-form}.
We may then view $x\in\{0,1\}^{N}$ as $x\in\{0,1\}^{N_{1}}\times\{0,1\}^{N_{2}}$
where $N_{1}=N\rho_N$. We call $x$ the configuration, and $x_i$ the spin of the $i$
particle. We call $I_{1}=[N_{1}]$ the \emph{first species} of particles
and $I_{2}=[N]\backslash I_{1}$ the \emph{second species} of particles.
We may then view the log-likelihood, $H_N$, as the \emph{Hamiltonian} for a \emph{two-species spin glass} model,
\[
H_{N}(x)=\frac{\sqrt{2}}{\sqrt{N}}\sum_{i,j}g_{ij}x_{i}x_{j},
\]
where here $g_{ij}$ are i.i.d. $\cN(0,1)$ (and in particular, $(g_{ij})$ is not symmetric).
Thus our goal is to compute the free energy of this two-species spin glass model constrained to certain classes
of configurations.

In the following, it will be useful to define the \emph{overlap} between two points $\sigma^1,\sigma^2\in\Sigma_N(\rho)$ by
\[
R(\sigma^{1},\sigma^{2})=\frac{1}{N}\sum_{i=1}^{N}\sigma_{i}^{\ell}\sigma_{i}^{\ell}. 
\]
When the notation is unambiguous we will also denote $R(\sigma^{i},\sigma^{j})=R_{ij}$. (In particular,  we let $R_{11} = R(\sigma^{1},\sigma^{1})$).
It will also be useful to define the \emph{intraspecies overlaps},
\begin{align*}
R^{(1)}(\sigma^{1},\sigma^{2})  =\frac{1}{N_{1}}\sum_{i=1}^{N_{1}}\sigma_{i}^{1}\sigma_{i}^{2},\,\,\,\,\,\,\,
R^{(2)}(\sigma^{2},\sigma^{2})  =\frac{1}{N_{2}}\sum_{i=N_{1}+1}^{N}\sigma_{i}^{1}\sigma_{i}^{2}.
\end{align*}
Note that for $i,j=1,2$, 
we have $R_{ij}=\rho_N R_{ij}^{(1)}+(1-\rho_N)R_{ij}^{(2)}$.

\subsection{The Upper Bound}
\label{sec:pf_ub}

Let $\cA_{r}$ denote the rooted tree of depth $r$ where each non-leaf
vertex has countably many children, i.e., the first $r$ levels of the Ulam-Harris tree.
Note that we may view the leaves of
this tree as the set $\N^{r}$, where $\alpha=(\alpha_{1},\ldots,\alpha_{r})$
denotes the root leaf path $\emptyset\to\alpha$. We denote this path by $p(\alpha)$. For more on this
notation see Appendix~\ref{app:RPC}.

\begin{thm}\label{thm:pf-ub}
Suppose that $(\rho,q)$ is such that $\Sigma_N(\rho,q)$ is non-empty. Then for $N\geq1$, we have that 
\begin{equation}\label{eq:ub}
\E F_{N}(\beta,\rho,q)\leq\inf_{\mu,\Lambda}P_{\beta}(\mu,\Lambda;q).
\end{equation}
\end{thm}
\begin{proof}
Let us begin with the case where $\mu$ has finite support.
To see this, let $(v_{\alpha})_{\alpha\in\mathbb{N}^{r}}$ denote
a Ruelle probability cascade (RPC) corresponding to parameters $(\mu_{\ell})$
(For a definition of Ruelle probability cascades see Appendix~\ref{app:RPC}.). Let $(Z(\alpha))$ denote
the centered Gaussian process on $\N^{r}$ with covariance 
\[
\E Z(\alpha)Z(\gamma)=4\beta^2Q_{\abs{\alpha\wedge\gamma}},
\]
where $\abs{\alpha\wedge\gamma}$ denotes the depth of the least common
ancestor of $\alpha$ and $\gamma$ in $\cA_{r}$, and let $(Z_{i}(\alpha))$
denote i.i.d. copies of this process. Similarly let $(Y(\alpha))$ be
the centered Gaussian process on $\N^{r}$ with covariance 
\[
\E Y(\alpha)Y(\gamma)=2 \beta^2 Q_{\abs{\alpha\wedge\gamma}}^{2}.
\]
Take the processes $Y$ and $Z$ to be independent of each other and $H_N$.
For $t\in[0,1]$, define the interpolating Hamiltonian $H_{t}(\sigma,\alpha):\Sigma_{N}(\rho,q)\times\N^{r}\to\R$
given by 
\[
H_{t}(\sigma,\alpha)=\sqrt{t}(\beta H_{N}(\sigma)+\sqrt{N}Y(\alpha))+\sqrt{1-t}\left(\sum_{i=1}^{N}Z_{i}(\alpha)\sigma_{i}\right).
\]
Finally let 
\[
\phi(t)=\frac{1}{N}\E\log\sum_{\alpha,\sigma}v_{\alpha}\exp(H_{t}(\sigma,\alpha)).
\]

\noindent
To control this, let us recall  Gaussian integration-by-parts for Gibbs measures (see, e.g., \cite[Lemma 1.1]{panchenkobook}):
\begin{lem*}
(Integration by parts for Gibbs measures) Let $\cX$ be at most
countable and let $(u(x)),(v(x))$ be centered Gaussian processes
on $\cX$ with mutual covariance $C(x,y)=\E u(x)v(y).$ Let $G$ be
the probability measure with $G(\{x\})\propto\exp(v(x))$. We have
the identity
\begin{equation}\label{eq:IBP}
\E\int_{{\cX}}u(x)dG(x)=\E\int\int_{\mathcal{X}^{2}}C(x^{1},x^{1})-C(x^{1},x^{2})dG^{\tensor2}.
\end{equation}
\end{lem*}

\noindent
If we let $(\sigma^{\ell},\alpha^{\ell})$ denote independent draws from
the Gibbs measure 
\[
\pi_{t}(\{\sigma,\alpha\})\propto\exp(H_{t}(\sigma,\alpha)),
\]
then integrating by-parts, we have that
\begin{align*}
\phi'(t) & =\frac{1}{N}\E\left\langle \partial_{t}H_{t}(\sigma,\alpha)\right\rangle =\frac{1}{N}\E\left\langle C_{11}-C_{12}\right\rangle 
\end{align*}
where
\begin{align*}
C_{ij} & =\E\partial_{t}H_t(\sigma^{i},\alpha^{i})\cdot H_t(\sigma^{j},\alpha^{j}) = \beta^2 (R_{ij}-Q_{\abs{\alpha^{i}\wedge\alpha^{j}}})^{2}\quad i,j=1,2.
\end{align*}
In the case that $i=j$ observe that, by definition of $\Sigma_{N}(\rho,q)$, $R_{11}=\rho=Q_{r}$ so that $C_{11}=0$.
Thus $\phi'\leq0$. The result then follows by comparing the boundary
conditions. In particular, re-arranging the inequality $\phi(1)\leq\phi(0)$
yields
\begin{equation}\label{eq:ub-1}
 \E F_{N}\leq\frac{1}{N}\E\log\sum_{\alpha}v_{\alpha}\sum_{\sigma\in\Sigma_{N}(\rho,q)}\exp(\sum_{i}Z_{i}(\alpha)\sigma_{i})-\frac{1}{N}\E\log\sum_{\alpha}v_{\alpha}\exp(\sqrt{N}Y(\alpha)),
\end{equation}
By Corollary~\ref{cor:exp-x-term}, the second term is equal to
\[
\frac{1}{N}\E\log\sum_{\alpha}v_{\alpha}e^{\sqrt{N}Y(\alpha)}=\frac{1}{2}\int 4\beta^2 s\mu([0,s])ds.
\]
It remains to upper-bound the first term. 

To this end, observe that the set $\Sigma_{N}(\rho,q)$
is defined via a constraint on the intra-species overlaps. If we add
 Lagrange multipliers, $\Lambda_1$ and $\Lambda_2$, for the constraints that $R_{11}^{(1)}=\rho$
and $R_{11}^{(2)}=(\rho-q)$, the first term in \eqref{eq:ub-1} is equal to 
\begin{align}
&\frac{1}{N}\E\log\sum_{\alpha}v_{\alpha}  \sum_{\sigma\in\Sigma_{N}(\rho,q)}\exp\Big(\sum_{i}Z_{i}(\alpha)\sigma_{i}+\Lambda_{1}N_{1}R_{11}^{(1)}+\Lambda_{2}N_{2}R_{11}^{(2)} \Big)-\Lambda_{1}q-\Lambda_{2}(\rho-q)\nonumber\\
 & \leq\frac{1}{N}\E\log\sum_{\alpha}v_{\alpha}\sum_{\sigma\in\Sigma_{N}}\exp\Big(\sum_{i}Z_{i}(\alpha)\sigma_{i}+\Lambda_{1}N_{1}R_{11}^{(1)}+\Lambda_{2}N_{2}R_{11}^{(2)} \Big)-\Lambda_{1}q-\Lambda_{2}(\rho-q),\label{eq:decoupled-guerra}
\end{align}
where the inequality follows by the set containment $\Sigma_{N}(\rho,q)\subseteq\Sigma_{N}=\{0,1\}^N$.
Observe that the summation in $\sigma$ is now over a product space, 
and that $R_{11}^{(\ell)} = \frac{1}{N_1}\sum_{i\leq N_1} \sigma^\ell_i$, since $\sigma_i\in\{0,1\}$. 
Thus the first term in this display is of the form 
\[
\frac{1}{N}\E\log\sum_{\alpha}v_{\alpha}\prod_{i=1}^{N_1}(1+\exp( Z_{i}(\alpha)+\Lambda_{1}))
\prod_{i=N_1+1}^{N} (1+\exp( Z_i(\alpha) + \Lambda_2)).
\]
Applying (both parts of) Theorem~\ref{thm:RPC-to-PDE} we see that
this is given by 
\begin{equation}\label{eq:guerra-to-pde}
\rho\E\log\sum_{\alpha}v_{\alpha}(1+\exp( Z(\alpha)+\Lambda_{1}))+(1-\rho)\E\log\sum v_{\alpha}(1+\exp\left( Z(\alpha)+\Lambda_{2}\right))=\rho \varphi_{\mu}^{1}(0,0)+(1-\rho)\varphi_{\mu}^{2}(0,0),
\end{equation}
where here we used that $N=N_{1}+N_{2}=\rho N+(1-\rho)N,$ and $\varphi_{\mu}^i=v_i$
for $v_i$ as in \eqref{eq:v-def-2}.
This gives the desired upper bound,
\begin{equation}\label{eq:upper-bound-no-min}
\E F_N(\beta,\rho,q) \leq P_\beta(\mu,\Lambda,q)
\end{equation}
 for $\mu$ of finite support. 

We obtain the result for general $\mu$ by continuity. Specifically, recall that $u_\mu^i$ admits continuous dependence on $\mu$ (see \cite[Sec 2.4]{jagannath2016tobasco}) and the last term in the definition of $P_\beta$ is a bounded linear functional of $\mu$.(Alternatively, we may use \prettyref{lem:lipschitz} below for both terms.) 
Thus $P_\beta$ weak-* continuous in $\mu$.
Since the set of measures of finite support is weak-* dense in $\cM_1([0,\rho])$, we obtain \eqref{eq:upper-bound-no-min} 
for all $\mu\in\cM_1([0,\rho])$ and $\Lambda$. Minimizing yields the desired bound.
\end{proof}

\subsection{Lower bound via Aizenman-Sims-Starr scheme}
\label{sec:pf_lb}

It now remains to prove the matching lower bound. In the following let
$Z_{N}(E)=\sum_{\sigma\in E}\exp(H(\sigma)).$ We begin with the following
inequality: for any $M\geq 1$, 
\[
\liminf_{N \to \infty} \frac{1}{N}\E\log Z_{N}(\Sigma_{N}(\rho_N,q_N))\geq\liminf_{N \to \infty}\frac{1}{M}\left[\E\log Z_{N+M}(\Sigma_{N+M}(\rho_N,q_N))-\E\log Z_{N}(\Sigma_{N}(\rho_N,q_N))\right].
\]
We call these new $M$ coordinates \emph{cavity coordinates}. Let us
take $M=M_{1}+M_{2}$ where we add $M_{1}$ cavity coordinates to the
first species and $M_{2}$ to the second. 
Throughout the following we will take the admissible sequence to be such that 
\begin{equation}\label{eq:rho-choice}
\abs{\rho_N-\rho},\abs{q_N-q}\leq \frac{C}{N}.
\end{equation}
(The bound for $q_N$ follows by definition of admissibility.)

Let us now decompose the Hamiltonian into the part induced by the
cavity and the remaining. We set
\begin{align*}
H_{N}'(\sigma) & =\frac{\sqrt{2}\beta}{\sqrt{N+M}}\sum_{i,j=1}^{N}g_{ij}\sigma_{i}\sigma_{j},\\
z_{N,i}(\sigma) & =\frac{\sqrt{2}\beta}{\sqrt{N+M}}\sum_{j=1}^{N}(g_{ij}+g_{ji})\sigma_{j},\\
y_{N}(\sigma) & =\frac{\sqrt{2}\beta}{\sqrt{N(N+M)}}\sum_{i,j=1}^Ng_{ij}'\sigma_{i}\sigma_{j},
\end{align*}
where the $g'_{ij}$ are independent of the $g_{ij}$. 
It then follows that
\begin{align*}
\beta H_{N+M}(\sigma) & =H_{N}'(\sigma)+\sum_{i=N+1}^{N+M}z_{N,i}(\sigma)\sigma_{i}+r_{1}(\sigma)\\
\beta H_{N}(\sigma) & =H_{N}'(\sigma)+\sqrt{M}y_{N}(\sigma)+r_{2}(\sigma),
\end{align*}
where $r_{i}$ is a centered Gaussian process with $\Var(r_{i})=O(\frac{1}{N})$. Then for $J=\{N+1,\ldots,N+M\}$,
\begin{align}
\E\log Z_{N+M}(\Sigma_{N+M}(\rho_{N+M},q_{N+M})) & =\E\log\sum_{\sigma\in\Sigma_{N+M}(\rho_{N+M},q_{N+M})}\exp(H'_{N}(\sigma))\cdot\prod_{i\in J}\exp(z_{N,i}(\sigma)\sigma_{i})+o(1)\label{eq:z-int}\\
\E\log Z_{N}\left(\Sigma_{N}(\rho_N,q_N)\right) & =\E\log\sum_{\sigma\in\Sigma_{N}(\rho_N,q_N)}\exp(H_{N}'(\sigma)+\sqrt{M}y_{N}(\sigma))+o(1),\label{eq:y-int}
\end{align}
where we have eliminated the $r_{i}$ dependence on the right hand
side by the following lemma. 
\begin{lem}\label{lem:interpolation-argument}
Let $\mathcal{X}$ be a finite set. Let $r:\cX\to\R$ be a centered
Gaussian process on $\cX$ and let $H(x)$ be a Gaussian process independent
of $r$. Then 
\[
|\E\log\sum_{x\in\mathcal{X}}\exp(H(x)+r(x))-\E\log\sum_{x\in\mathcal{X}}\exp(H(x))|\leq Var(r).
\]
\end{lem}

\begin{proof}
Let $Y_{t}(x)=H(x)+\sqrt{t}r(x)$ and 
\[
\psi(t)=\E\log\sum\exp Y_{t}.
\]
Then if we let $G_{t}\in Pr(\mathcal{X})$ be $G_{t}(\{x\})\propto e^{Y_{t}}$
then we have 
\[
\partial_{t}\psi(t)=\E\int\partial_{t}Y_{t}dG_{t}=\E\int C(x^{1},x^{1})-C(x^{1},x^{2})dG^{\tensor2}
\]
where 
\[
C(x,y)=\E\partial_{t}Y_{t}(x)Y_{t}(y)=\frac{1}{2}\E r_{t}(x)r_{t}(y)\leq\frac{1}{2}Var(r).
\]
Consequently $\abs{\partial_{t}\psi(t)}\leq Var(r).$ From which it
follows that $\abs{\psi(1)-\psi(0)}\leq Var(r).$ Evaluating $\psi$
at $0$ and $1$ yields the desired bound.
\end{proof}

\subsubsection{Reduction to continuous functionals}
Our goal  now is to compute the limit of the difference of \eqref{eq:z-int} and \eqref{eq:y-int}. 
We begin by showing that their difference can be related to the difference
of two continuous functionals on an appropriate space of probability measures.
To this end, we first claim that, upon passing to a subsequence in $N$, we may choose 
a sequence $(r_{M},u_M)\to (\rho,q)$ such that 
\[
\Sigma_{N+M}(\rho_{N+M},q_{N+M})\supset\Sigma_{N}(\rho_{N},q_{N})\times\Sigma_{M}(r_{M},u_{M}) 
\]
where here we have abused notation and 
view the cavity coordinates $\Sigma_{M}=\Sigma_{M_{1}}\times\Sigma_{M_{2}}$  in this product
as being distributed between the species in such a way that $N_{1}+M_{1}$
is the size of the first species. 

To see that such a sequence exists, observe that if we define $r_M(N)$ and $u_M(N)$ by
\begin{align*}
Mr_M(N) &= (N+M)\rho_{N+M} - N\rho_N\\
Mu_M(N) &= (N+M)q_{N+M} - Nq_N,
\end{align*}
then by the choice of the sequence $(\rho_N,q_N)$ (see \eqref{eq:rho-choice}),
these are both integers bounded by $M$. As such, we may pass to a subsequence along which 
the pair converges. In particular, eventually along this subsequence in $N$, $(M r_M,M u_M)$
will be constant. The desired properties of $(r_M,u_M)$ follows immediately by definition of $(\rho_N,q_N)$.
 
Using this sequence, we may lower bound \eqref{eq:z-int} by
\begin{equation}\label{eq:z-int-lb}
\begin{aligned}
\E\log\sum_{\sigma\in\Sigma_{N+M}(\rho_{N+M,}q_{N+M})}&\exp(\beta H_{N+M}(\sigma))\\ & \geq\E\log\sum_{\sigma\in\Sigma_{N}(\rho_{N},q_{N})}\exp(H'(\sigma))\cdot\sum_{\epsilon\in\Sigma_{M}(r_{M},u_{M})}\prod_{i\in J}\exp(z_{N,i}(\sigma)\epsilon_{i})+o(1).\\
\end{aligned}
\end{equation}
Furthermore, since 
\[
\E y_{N}(\sigma^1)y_{N}(\sigma^2) =\frac{2\beta^2}{N(N+M)}\left(\sigma^{1}\cdot\sigma^{2}\right)^{2}=2\beta^2\frac{N}{N+M}R_{12}^{2}=2\beta^2R_{12}^{2}+o(1),\\
\]
if we let $(y'(\sigma))$ denote the Gaussian process with covariance
\[
\E y'(\sigma^1)y'(\sigma^2) = 2\beta^2 R_{12}^2,
\]
then we may apply Lemma~\ref{lem:interpolation-argument} to express \eqref{eq:y-int} as
\begin{equation}\label{eq:y-int-2}
\E\log Z_{N}\left(\Sigma_{N}(\rho_N,q_N)\right)  =\E\log\sum_{\sigma\in\Sigma_{N}(\rho_N,q_N)}\exp(H_{N}'(\sigma)+\sqrt{M}y'(\sigma))+o(1).
\end{equation}

\noindent
Define $G'$ to be the Gibbs distribution corresponding to $H'$ on the subset $\Sigma_N(\rho_N,q_N)$,
$$G'(\{\sigma\})\propto \exp(H'(\sigma))\indicator{\sigma\in\Sigma_N(\rho_N,q_N)},$$ 
where we keep the dependence on $N$ implicit. Let $\gibbs{\cdot}_{G'}$ denote
expectation with respect to this measure. Subtracting \eqref{eq:y-int-2} from \eqref{eq:z-int-lb}, we may lower bound 
the difference of \eqref{eq:z-int}  and \eqref{eq:y-int} by
\[
\E\log\left\langle \sum_{\epsilon\in\Sigma_{M}(r_{M},u_{M})}\exp(\sum_{i}z_{N,i}(\sigma)\epsilon_{i})\right\rangle _{G'} -\E\log\left\langle \exp(\sqrt{M}y_{N}(\sigma))\right\rangle _{G'},
\]
up to a $o(1)$ correction. We now provide alternative representations for these two terms. 

To this end, let $\sR$ be the space of infinite Gram arrays with entries bounded by $1$. Let $\cM_{exch}$ be the subspace of probability measures on $\sR$
that are \emph{weakly exchangeable}, i.e., if $R$ is drawn from some $\cR\in\cM_{exch}$, then
\[
(R_{\ell\ell'}) \eqdist (R_{\pi(\ell)\pi(\ell')}),
\]
for any permutation $\pi:\N\to\N$ which permutes finitely many numbers. 

By standard arguments, one can show that there
are weak-* continuous functionals $\cF_{i,M}:\cM_{exch}\to\R$ such that if we let $\cR_N$ denote the law
of the Gram array formed by the overlaps of $(\sigma^\ell)\sim (G')^{\tensor \infty}$---called the \emph{overlap distribution} corresponding to $G'$---then 
\begin{align*}
\cF_{1,M}(\cR_N) &= \frac{1}{M}\E\log\left\langle \sum_{\epsilon\in\Sigma_{M}(r_{M},u_{M})}\exp(\sum_{i}z_{N,i}(\sigma)\epsilon_{i})\right\rangle _{G'}\\
\cF_{2,M}(\cR_N) &= \frac{1}{M}\E\log\left\langle \exp(\sqrt{M}y_{N}(\sigma))\right\rangle _{G'},
\end{align*}
This follows from the following 
 general continuity theorem for functionals of this form. 

Let $S\subset\Sigma_{M}(\rho)$, and let $(w_{\alpha})_{\alpha\in\mathscr{A}}$
be the weights of a (possibly random) probability measure on a countable set $\mathscr{A}.$
Denote the measure by $\Gamma(\{\alpha\})=w_{\alpha}$. Let $R_{\mathscr{A}}=(R_{\alpha^{1},\alpha^{2}})_{\alpha^{1},\alpha^{2}\in\mathscr{A}}$
be a doubly infinite, positive semidefinite matrix. We think of the matrix $R_{\mathscr{A}}$
 as the collection of allowed values for an overlap and call it an \emph{abstract overlap
structure} (defined by $w_{\alpha}$). 

Consider the functionals 
\begin{align}
f_{1,M} & =\frac{1}{M}\E\log\left[\sum_{\alpha\in\sA}w_{\alpha}\sum_{\sigma\in S}\exp\left(\sum_{i=1}^{M}Z_{i}(\alpha)\sigma_{i}\right)\right]\label{eq:f1-abstract}\\
f_{2,M} & =\frac{1}{M}\E\log\left[\sum_{\alpha\in\sA}w_{\alpha}\exp(\sqrt{M}Y(\alpha))\right].\label{eq:f2-abstract}
\end{align}
Here $Z_{i}(\alpha)$ are iid copies of a centered Gaussian process
$Z(\alpha)$ with covariance 
\[
Cov(Z(\alpha^{1}),Z(\alpha^{2}))=C_{Z}(R_{\alpha^{1}\alpha^{2}}),
\]
for some continuous function $C_{Z}.$ Similarly $Y(\alpha)$ is a
centered Gaussian process with covariance 
\[
Cov(Y(\alpha^{1}),Y(\alpha^{2}))=C_{Y}(R_{\alpha^{1}\alpha^{2}}),
\]
for some continuous function $C_{Y}$. 

Let $(\alpha(\ell))_{\ell\geq1}$ be iid draws from $\Gamma$, and
for $n\geq1$ consider the random matrix 
\[
R^{n}=\left(R_{\alpha(\ell),\alpha(\ell')}\right)_{\ell,\ell'\in[n]}.
\]
We then have the following result from \cite[Lemma 8]{panchenko2017potts}.
\begin{thm*} For any $\epsilon>0$ there are continuous bounded functions
$g_{\epsilon}^{Z},g_{\epsilon}^{Y}:\R^{n^{2}}\to\R$ such that 
\[
\abs{f_{1,M}-\E g_{\epsilon}^{Z}(R^{n})}\leq\epsilon\quad\abs{f_{2,M}-\E g_{\epsilon}^{Y}(R^{n})}\leq\epsilon.
\]
Furthermore, these functions depend only on $M,S,C_{Z},C_{Y},\rho,$ and
$\epsilon$. 
\end{thm*}
Before continuing we note here that since the diagonal terms of $R^n$ are always $\rho$,  this functional only depends
on the diagonal through $\rho$.
Let us now show that we can slightly modify $\cR_N$ in such a way that we can compute these limits explicitly.

\subsubsection{Perturbation of overlap distribution}\label{sec:perturb}

Before computing this limit, we begin by observing that we may assume,
up to a perturbation, that these measures satisfy what is called the
\emph{Ghirlanda-Guerra identities} which are defined as follows. 
Let $R=(R_{\ell\ell'})_{\ell,\ell'\geq 1}$ satisfy $R\sim\cR$ for some $\cR\in\cM_{exch}$. 
We call such a matrix a \emph{Gram-de Finetti array}.
\begin{defn}
Let $R=(R_{\ell,\ell'})$ be a Gram-de Finnetti array with
$\abs{R_{ij}}\leq1$. We say that the law of $(R_{\ell,\ell'})$ satisfies
the \emph{Ghirlanda-Guerra identitites }if for every $n\geq1$, $f\in L^{\infty}(\R^{n^{2}})$
and $p\geq1$,
\[
\E\ f(R^{n})\cdot R_{1,n+1}^{p} =\frac{1}{n}\left[\E f(R^{n}) \cdot\E R_{12}^{p} +\sum_{\ell=2}^{N}\E f(R^{n})\cdot R_{1\ell}^{p} \right].
\]
\end{defn}

We begin by observing that we may perturb the Hamiltonian $H$ so
that it satisfies the Ghirlanda-Guerra identities.
\begin{lem}\label{lem:pert}
Let 
\[
h_{N}(\sigma)=\sum\frac{x_{p}}{2^{p}}g_{p}(\sigma)
\]
where $(x_{p})\in[1,2]^{\N}$, and 
\[
g_{p}(\sigma)=\frac{1}{N^{\frac{p}{2}}}\sum g_{i_{1}\ldots i_{p}}\sigma_{i_{1}}\cdots\sigma_{i_{p}}
\]
where $(g_{i_{1}\ldots,i_{p}})$ are i.i.d. and independent of $W.$
Let $s_{N}\to\infty$ with $N^{1/4}\ll s_{N}\ll N^{1/2}$. There is
a sequence of choices of parameters $(x_{p}^{N})$ such that 
\begin{equation}\label{eq:perturb}
\lim\abs{\frac{1}{N}\E\log\sum\exp(\beta H_N(\sigma)+s_{N}\beta h_{N}(\sigma))-\frac{1}{N}\E\log\sum\exp(\beta H(\sigma))}=0
\end{equation}
and such that if $R$ denotes the limiting law of the Gram array formed
by overlaps of i.i.d. samples from the Gibbs measure $\pi_{pert}(\{\sigma\})\propto\exp(-\beta H_N(\sigma)-\beta h(\sigma))$,
then $R$ satisfies the Ghirlanda-Guerra identities.
\end{lem}
\begin{proof}[Proof Sketch ]
Results of this type are standard in the spin glass literature. For a textbook presentation see
\cite[Chap. 3]{panchenkobook}. We only sketch the key points here and how they differ from 
\cite{panchenkobook}. 

For any sequence of choices of parameters, by the condition on $s_{N}$, \eqref{eq:perturb}
holds by an application of Jensen's inequality. One can then
show that if we choose the parameters $(x_{p})$ to be drawn i.i.d.
from the uniform measure on $[1,2]$, then
\[
\E_{x}\E\left\langle \abs{(g_{p}-\E\left\langle g_{p}\right\rangle )}\right\rangle \to0.
\]
Consequently for any choice of $n\geq1$ and $f\in L^{\infty}(\R^{n}),$
we obtain, conditionally on $x$,
\[
\E\left\langle f(R^{n})g_{p}\right\rangle =\E\left\langle f(R^{n})\right\rangle \cdot\E\left\langle g_{p}\right\rangle +o(1).
\]
Applying Gaussian integration, we obtain
\[
\E\left\langle f(R^{n})(R_{11}^{p}+R_{12}^{p}+\cdots+R_{1n}^{p}-(n+1)R_{1n+1}^{p}\right\rangle =\E\left\langle f\right\rangle \E\left\langle R_{11}^{p}-R_{12}^{p}\right\rangle,
\]
which yields, upon re-arrangement
\[
\frac{1}{n}\left(\E\left\langle fR_{11}^{p}\right\rangle -\E\left\langle f\right\rangle \E\left\langle R_{11}^{p}\right\rangle \right)+\frac{1}{n}\left(\E\left\langle f\right\rangle \E\left\langle R_{12}^{p}\right\rangle +\sum_{i=2}^{n}\E\left\langle f\cdot R_{1i}^{p}\right\rangle \right)=\E\left\langle fR_{1n+1}^{p}\right\rangle +o(1).
\]
The main issue in settings where the Gibbs measure is not on the discrete
hypercube $\{-1,1\}^{N}$ are the terms with the self-overlap, $ R_{11}$.
This, however, is not an issue in our setting as $R_{11}$ is constant, so that
the first term in the above vanishes identically for all $N$. 

The remaining argument is then unchanged from \cite[Sec 3.2]{panchenkobook}. In particular,
by this vanishing, we obtain that the error between the left and right hand sides of
the Ghirlanda-Guerra identity for any $p$ by this argument. The existence of a suitable sequence 
then follows by the probabilistic method as in \cite[Lemma 3.3]{panchenkobook}. 
\end{proof}
As a consequence of this, it suffices to evaluate the limits of $\cF_{i,M}$ on the overlap distribution, $\cR_N'$, 
corresponding to the perturbed Gibbs distribution
$G'_{pert,N}(\{\sigma\})\propto\exp(\beta H_N(\sigma)+\beta h_{N}(\sigma))$,  where $h_{N}$ is obtained form the preceding lemma 
(and we still restrict to $\Sigma_N(\rho_N,q_N)$). As $\sR$ is compact, we see that any weak-* limit point of this 
sequence satisfies the Ghirlanda-Guerra identities.  Let us now briefly recall the structure of the subspace of $\cM_{exch}$ which satisfy these identities.

To this end, we first recall the construction of overlap distributions corresponding to Ruelle probability cascades.
Fix $r\geq1$ and a pair of sequences 
\begin{align*}
0 & =Q_{0}<\ldots<Q_{r}=\rho\\
0 &=\mu_{-1}<\mu_0<\ldots<\mu_r = 1.
\end{align*}
Let $(v_{\alpha})_{\alpha\in\N^{r}}$ denote the weights of a
Ruelle probability cascade corresponding to the parameters $(\mu_\ell)_{\ell=-1}^r$.
Corresponding to this sequence, we may
define the (random) probability measure on the ball of $\ell_{2}$
of radius $\rho$ given by
\begin{equation}\label{eq:pi-def}
\begin{aligned}
\pi & =\sum v_{\alpha}\delta_{s_{\alpha}}\\
s_{\alpha} & =\sum_{\gamma\in p(\alpha)}\sqrt{Q_{\abs{\gamma}}-Q_{\abs{\gamma}-1}}e_{\gamma},
\end{aligned}
\end{equation}
where $(e_{\alpha})_{\alpha\in\cA_{r}\backslash\{\emptyset\}}$.
Let $(\sigma^\ell)\sim \pi^{\tensor\infty}$ and let 
$\cR$ denote the law of their corresponding Gram matrix.
If we let $\mu\in\cM_1([0,\rho])$ be defined by $\mu(\{Q_\ell\})=x_\ell$,
then by \prettyref{lem:overlap-rpc} $\mu$ is the law of the first off-diagonal entry of the gram array. 
In particular, $\mu$ is a sufficient statistic for the family of such $\cR$ and we will denote this law by $\cR(\mu)$.

Let $\cM_f\subseteq\cM_1([0,\rho])$ be those measures of finite support.
As a consequence of Panchenko's Ultrametricity theorem \cite{panchenkoult13} and
the Baffioni-Rossati theorem \cite{baffioni2000some} and \cite[Theorem 15.3.6]{TalagrandMFSGvol2},
the space of Gram-de Finnetti arrays is
given by the closure of the set $\mathscr{G}=\{ \cR(\mu): \mu\in\cM_f\}$ in the weak-* topology.
We  denote a law of this type by
$\cR(\mu)$. In particular, if $(\cR_n)\subset \mathscr{G}$ is a sequence
of  laws in this closure, then the law of the off-diagonals of $\cR_n$ converge to the law
of the off diagonals of $\cR$ if and only if $\mu_n(\cdot) = \cR_n(R_{12} \in \cdot)$
have $\mu_n\to\mu$. For a precise statement of these two results, see \cite[Theorem 2.13, Theorem 2.17]{panchenkobook}.
(As we will always be using this result in the case that the diagonals are constant, this convergence is
will be effectively for the entire array for our purposes.)

Thus it suffices to compute these limiting functionals
on Ruelle cascades and then, take their limits (provided they exist). To carry out this program,
it will be important to note that, restricted to the space of Ruelle probability cascades,
these functionals are in fact Lipschitz. 

\subsubsection{Lipschitz Continuity}
Let us restrict our attention to functionals of the form $\cF_{i,M}$ on the 
space of Ruelle probability cascades. 
If we let $\cR$ denote the law of the overlap array corresponding to a Ruelle cascade as above, then by construction
\begin{align*}
\cF_{1,M}(\cR) &= \frac{1}{M}\E \log \sum v_\alpha \sum_{\sigma\in\Sigma_M(\rho,q)}\exp( \sum_i Z_i(\alpha)\sigma_j)\\
\cF_{2,M}(\cR) &=\frac{1}{M} \E \log \sum v_\alpha \exp(\sqrt M Y(\alpha)).
\end{align*}
Let us now understand these functionals in more detail. 

To this end,  let 
\[
R_{\N^{r}}=(s_{\alpha}\cdot s_{\alpha'})_{\alpha,\alpha'\in\mathbb{N}^{r}}
\]
denote the collection of overlaps defined by $s_{\alpha}$ from \eqref{eq:pi-def}, we may
define the functional $f_{1,M}(\mu;S)$ to be the functional as in
\eqref{eq:f1-abstract} with weights given by the $v_{\alpha}$ and abstract
overlap structure $R_{\mathbb{N}^{r}}$ for some choice $S\subset\Sigma_{M}(\rho)$.
Define $f_{2,M}(\mu;S)$ similarly with \eqref{eq:f2-abstract}. 
In this case, if we let $\cR$ denote such an overlap distribution and let $\mu$ denote the
law of the first off diagonal, observe that we have 
\begin{align*}
f_{i,M}(\mu;S)=\cF_{i,M}(\cR) \quad i=1,2.
\end{align*}
In particular, this functional depends on $\cR$ only through $\mu$.
Let us study the regularity of the map $\mu\mapsto f_{1,M}(\mu;S)$.

To this end, equip $\cM_{1}([0,\rho])$
with the the Kantorovich metric, 
\[
d(\mu,\nu)=\int_{0}^{1}\abs{\mu^{-1}(s)-\nu^{-1}(s)}ds=\int_{0}^{\rho}\abs{\mu([0,s])-\nu([0,s])},
\]
where for a probability measure $\mu,$ we let $\mu^{-1}$ denote
its quantile function. Recall that $d$ metrizes the weak-{*} topology
on $\cM_{1}([0,\rho]).$ We then have the following.
\begin{lem}\label{lem:lipschitz}
The following holds for any $M\geq1$.
\begin{enumerate}
\item For any $S\subset\Sigma_{M}(\rho),$ the map $\mu\mapsto f_{1,M}(\mu;S)$
is Lipschitz (uniformly in $M,\rho$).
\item The map $\mu\mapsto f_{2,M}(\mu)$ is Lipschitz (uniformly in $M,\rho$).
\end{enumerate}
In particular, these functionals are well-defined and Lipschitz  (uniformly in $M,\rho$) on all of $\cM_{1}([0,\rho])$.
\end{lem}

\begin{proof}
We begin with the first claim. To this end, fix two measures $\mu,\tilde{\mu}$
of finite support. Consider their quantile functions $\mu^{-1}$ and
$\tilde{\mu}^{-1}$. Observe that we may view these as monotone paths
from $0$ to $\rho$ indexed by $[0,1]$. In particular, we may parameterize
these paths (and thus the measures) as follows. Since these paths
have finitely many jumps,  there is some $r\geq 1$ and some sequence
\[
0=x_{-1}<x_{0}<...<x_{r}=1
\]
such the jumps in either path are given by an increasing subset of
the times $(x_{\ell}).$ Furthermore, we may parameterize the supports
of $\mu$ and $\tilde{\mu}$ by two finite sequences 
\begin{align*}
0 & =Q_{0}\leq...\leq Q_{r}=\rho\\
0 & =\tilde{Q}_{0}\leq\cdots\leq\tilde{Q}_{r}=\rho,
\end{align*}
 (here we allow repetition) that satisfy $Q_{k}=\mu^{-1}(x_{k})$
and $\tilde{Q}_{k}=\tilde{\mu}^{-1}(x_{k})$. 

Let $(v_{\alpha})$
denote the Ruelle probability cascade corresponding to $(x_{r}),$
\[
H_{t}(\epsilon,\alpha)=\sqrt{t}\sum Z_{i}(\alpha)\epsilon_{i}+\sqrt{1-t}\sum\tilde{Z}_{i}(\alpha)\epsilon_{i},
\]
and 
\[
\Phi(t)=\frac{1}{M}\E\log\sum_{\N^{r}}v_{\alpha}\sum_{\sigma\in S}e^{H_{t}(\sigma,\alpha)}.
\]
Since the measure induced by $x$ and the sequence $Q$ is $\mu$
and like wise for $x,\tilde{Q}$, and $\tilde{\mu}$, we have that
\[
\Phi(0)=f_{1,M}(\tilde{\mu};S)\quad\Phi(1)=f_{1,M}(\mu;S).
\]
Differentiating in time we find that 
\[
\Phi'(t)=\frac{1}{M}\E\left\langle \partial_{t}H_{t}(\sigma,\alpha)\right\rangle _{t},
\]
where $\left\langle \cdot\right\rangle _{t}$ denotes integration
with respect to the measure $G_{t}(\epsilon,\alpha)\propto v_{\alpha}\exp(H_{t}(\epsilon,\alpha)).$
Observe that 
\[
\frac{1}{M}\E\partial_{t}H_{t}(\epsilon,\alpha)H_{t}(\epsilon',\alpha')=2\beta^{2}\left(Q_{\abs{\alpha\wedge\alpha'}}-\tilde{Q}_{\abs{\alpha\wedge\alpha'}}\right)R(\epsilon,\epsilon').
\]
Since $\abs{R(\epsilon,\epsilon')}\leq\rho$, and $Q_{r}=\tilde{Q}_{r}=1$,
we see by Gaussian integration-by-parts \eqref{eq:IBP} that 
\[
\abs{\Phi'(t)}\leq2\beta^{2}\rho\E\left\langle \abs{Q_{\alpha^1 \wedge\alpha^2}-\tilde{Q}_{\alpha^1\wedge\alpha^2}}\right\rangle_t,
\]
where $\alpha^1,\alpha^2$ are (the second terms of) two independent draws from $G_t$.

Now since the law of $H_{t}(\epsilon,\alpha)$ is invariant for $0\leq t\leq1$
we may apply the Bolthausen-Sznitman invariance, specifically \prettyref{thm:RPC-to-PDE},
to find that the marginal of $G_{t}$ on $\N^{r}$ is equal to $v_{\alpha}$
for all $t$. Consequently 
\begin{align*}
\E\left\langle \abs{Q_{\alpha^1 \wedge\alpha^2}-\tilde{Q}_{\alpha^1 \wedge\alpha^2}}\right\rangle  & =\E\sum_{\alpha,\alpha'}v_{\alpha}v_{\alpha'}\abs{Q_{\abs{\alpha\wedge\alpha'}}-\tilde{Q}_{\abs{\alpha\wedge\alpha'}}}
  =\sum_{1\leq k\leq r}\abs{Q_{k}-\tilde{Q}_{k}}\E\sum_{\abs{\alpha\wedge\alpha'}=k}v_{\alpha}v_{\alpha'}\\
 & =\sum_{1\leq k\leq r}\abs{Q_{k}-\tilde{Q}_{k}}\left(x_{k}-x_{k-1}\right)
  =\int_{0}^{1}\abs{\mu^{-1}(x)-\tilde{\mu}^{-1}(x)}dx
\end{align*}
where the second to last inequality follows by  \prettyref{lem:overlap-rpc}.
Combining this with the preceding yields 
\[
\abs{f_{1,M}(\mu;S)-f_{1,M}(\tilde{\mu},S)}\leq2\beta^{2}\rho d(\mu,\tilde{\mu}),
\]
which yields the desired since $\rho\leq 1$.
We now turn to the second claim. This follows by direct calculation.
Here we may apply \prettyref{thm:RPC-to-PDE} and Corollary \ref{cor:exp-x-term},
to find that
\[
f_{2,M}(\mu)=4\beta^{2}\int_{0}^{\rho}t^{2}\mu([0,t])dt,
\]
from which it follows that 
\[
\abs{f_{2,M}(\mu)-f_{2,M}(\tilde{\mu})}\leq4\beta^{2}\rho\int_{0}^{\rho}\abs{\mu([0,t])-\tilde{\mu}([0,t])}dt=4\beta^{2}\rho d(\mu,\tilde{\mu}).
\]
which yields the desired as  $\rho\leq 1$.
\end{proof}

\subsubsection{Computing limits}
Let $\mu$ denote the limiting law of $R_{12}$ with respect to $G'_{pert}$,
and let $\mu_{r}$ denote a sequence of discretization of $\mu$ with
finitely many atoms with $\mu_{r}\to\mu$ weak-{*} and $\mu_r(\{\rho\})>0$. 
Since since $\rho$ is always charged for this sequence, the overlap distirbution $\cR(\mu_r)$
has diagonal equal to $\rho$. 
Thus if $\cR$ is the limiting overlap distribution corresponding to $G'_{pert}$, then, as explained
at the end of \prettyref{sec:perturb}, the laws $\cR(\mu_r)\to\cR$ weak-* since $\mu$ determines
the off-diagonal of $\cR$ and the diagonals are constant an equal to $\rho$.

We first observe that $\cF_{2}$ may be handled as previously: 
we have that 
\begin{align*}
\lim_{N\to\infty}\cF_{2,M}(\cR_N')   = \lim_{r\to\infty}\frac{1}{M}\E\log\sum v_{\alpha}e^{\sqrt{M}Y(\alpha)}
  =\lim_{r\to\infty}\frac{1}{2}\int_{0}^{\rho}4\beta^2s\mu_{r}([0,s])ds=2\beta^2\int_{0}^{\rho}s\mu([0,s])ds,
\end{align*}
where the first equality follows by continuity of $\cF_{2,M}$, the second by Corollary~\ref{cor:exp-x-term}, and the final follows by weak-* continuity of bounded linear functionals.
It remains to consider the first term, $\cF_{1,M}$.

Let 
\[
A_{M}=\lim_{N\to\infty}\frac{1}{M}\E\log\left\langle \sum_{\epsilon\in\Sigma_{M}(r_{M},u_{M})}\exp(\sum_{i}z_{N,i}(\sigma)\epsilon_{i})\right\rangle _{G'_{pert}}.
\]
Then by \prettyref{lem:lipschitz}, it follows that 
\[
A_{M}\geq\frac{1}{M}\E\log\sum_{\alpha\in\N^{r}}v_{\alpha}\sum_{\epsilon\in\Sigma_{M}(r_{M},u_{M})}\exp\left(\sum_{i}Z_{i}(\alpha)\epsilon_{i}\right)-Kd(\mu_{r},\mu),
\]
where $(v_{\alpha})$ are the weights of the Ruelle probability cascade with parameters $(\mu_r)$, and $K$ does not depend on $M$. Recall from the proof of \prettyref{thm:pf-ub} that we were 
able to produce an upper bound for this term
where $\Lambda_i$ play the role of Lagrange multipliers for the constraints defining $\Sigma_M(r_M,u_M)$.
One can also produce a matching lower bound by minimizing over the choice of Lagrange multipliers whose proof is deferred to the next section. 
\begin{thm}\label{thm:decoupling}
Let $\mu\in\cM_{1}([0,\rho])$ have finite support.
Let $(\rho_{M},q_{M})\to(\rho,q)$ admissible. We have that 
\[
\lim_{M\to\infty}f_{1,M}(\mu,\Sigma_{M}(\rho_{M},q_{M}))=\inf_{\Lambda_{1},\Lambda_{2}}\rho\varphi_{\mu}^1(0,0)+(1-\rho)\varphi_{\mu}^2(0,0)-\Lambda_{1}q-\Lambda_{2}(\rho-q).
\]
\end{thm}
In particular, we have that, for each $r\geq 1$, 
\[
\lim_{M\to\infty}\cF_{1,M}(\mu_{r})=\lim_{M\to\infty}f_1(\mu_r,\Sigma_M(\rho_M,q_M))=\inf_{\Lambda_{1},\Lambda_{2}}
\rho\varphi_{\mu_r}^1(0,0)+(1-\rho)\varphi_{\mu_r}^2(0,0)-\Lambda_1 q-\Lambda_2(\rho-q).
\]
As the functionals $\mu\mapsto f_1(\mu,\Sigma_M(\rho_M,q_M))$ are uniformly Lipschitz 
 by \prettyref{lem:lipschitz}, so is their limit. In particular we see that 
\[
\lim_{M\to\infty}A_{M}\geq\inf_{\Lambda_{1},\Lambda_{2}}\rho\varphi_{\mu}^1(0,0)+(1-\rho)\varphi_\mu^2(0,0)
-\Lambda_1q -\Lambda_2(\rho-q)-2Kd(\mu_{r},\mu),
\]
for some $K$. 
If we then send $r\to\infty,$ we see that we have 
\begin{equation}\label{eq:lb}
\liminf_{N\to\infty}\tilde{F}_{N}\geq\inf_{\Lambda_{1},\Lambda_{2}}P_\beta(\mu,\Lambda_{1},\Lambda_{2}).
\end{equation}
this yields the desired lower bound.  

\begin{proof}[\textbf{\emph{Proof of \prettyref{prop:generalized-multispecies}}}]
Upon combining \eqref{eq:ub} with \eqref{eq:lb} we obtain \eqref{eq:free-energy-goal}. 
Recalling the equality $P_\beta = \beta \cP_\beta$ and dividing by $\beta$ yields the result,
except with an infimum. We then recall Lemma~\ref{lem:minimum} to conclude that this infimum is actually
achieved.
\end{proof}
\subsection{Removing constraints via Lagrange multipliers}
We now turn to the proof of Theorem~\ref{thm:decoupling}. 
The lemmas mentioned in this proof will be proved in the following section.
Let $\cS=\{(x,y)\in[0,1]^{2}:y\leq x\}$.
\begin{proof}[\textbf{\emph{Proof of Theorem~\ref{thm:decoupling}}}.]
Observe that as in \eqref{eq:decoupled-guerra} we were able
to obtain the corresponding upper bound. The question is then to show
that the infimum is achieved. We follow a strategy similar to \cite[Sec. 3]{panchenko2017potts}. 
In the following, it will be useful to recall that the quantities
we wish to compute are invariant under permutations of the entries of the vector $v$. 

For a point $\mathbf{x}=(\rho,q)\in\cS$ and $\epsilon>0$, consider
the set 
\[
\Sigma_{M}^{\epsilon}(\mathbf{x})=\Big\{\sigma\in\Sigma_{M}:\frac{1}{M}\sum_{i=1}^{M}\sigma_{i}\in[\rho-\epsilon,\rho+\epsilon],
\frac{1}{M}\sum_{i=1}^{\ceil{M\rho}} \sigma_i \in [q-\epsilon,q+\epsilon]\Big\}.
\]
We begin by observing that the error caused by this epsilon dilation is
negligible. Let $\cS_{M}\subseteq\cS$ be those $\mathbf{x\in\cS}$
such that $\Sigma_{M}(\mathbf{x})=\Sigma_{M}(x_{1},x_{2})$ is non-empty.
\begin{lem}
\label{lem:f1-to-f1-eps-bound} There is an $C>0$ such that for $\epsilon>\delta>0$
sufficiently small and any $M\geq1$ we have
\[
\begin{aligned}
\sup_{x\in\cS_{M}}\abs{f_{1,M}(\mu,\Sigma_{M}^{\epsilon}(\mathbf{x}))-f_{1}(\mu,\Sigma_{M}(\mathbf{x}))}\leq C\sqrt{\epsilon}\\
\sup_{x\in\cS}\abs{f_{1,M}(\mu,\Sigma_{M}^{\epsilon}(\mathbf{x}))-f_{1}(\mu,\Sigma^{\delta}_{M}(\mathbf{x}))}\leq C\sqrt{\epsilon-\delta}\\
\end{aligned}
\]
\end{lem}
Next, we begin by observing that on such sets, the limit of these functionals
exist and is concave.
\begin{lem}
\label{lem:f1-eps-concave}For each $\epsilon>0$, the limit
\[
f_{1}(\mu,\epsilon,\mathbf{z})=\lim_{M\to\infty}f_{1,M}(\mu,\Sigma_{M}^{\epsilon}(\mathbf{z}))
\]
exists and is concave. 
\end{lem}

Combining these claims we find that 
\begin{equation}
f_{1}(\mu;\rho,q)=\lim_{M\to\infty}f_{1,M}(\mu,\Sigma_{M}(\rho_{M},q_{M}))=\lim_{\epsilon\to0}f_{1}(\mu;\rho,q,\epsilon).\label{eq:f1-M-lim}
\end{equation}
exists and is concave in $(\rho,q)$. Furthermore, applying the first
claim again, we find that is $1/2$-H\"older. In particular, it is uniformly
continuous on $\cS$.

It remains for us to relate $f_{1}(\mu;\rho,q)$ to Parisi type functionals.
To this end, we observe the following. For any $\mathbf{\Lambda}=(\Lambda_{1},\Lambda_{2})\in\R^{2}$,
let
\[
F(\mathbf{\Lambda})=\rho \varphi_{\mu}^{1}(0,0;\Lambda_1)+(1-\rho)\varphi_{\mu}^{2}(0,0;\Lambda_2)-\Lambda_{1}q-\Lambda_{2}(\rho-q)
\]
(here we have made the dependence of $\varphi^i$ on $\Lambda_i$ explicit for clarity).
We then have the following.
\begin{lem}
\label{lem:RPC-exp-integral-compute} For any $(\Lambda_{1},\Lambda_{2})$,
we have that 
\begin{equation}
F(\mathbf{\Lambda})=\max_{\mathbf{x}\in\cS}\left\{ f_{1}(\mu,\mathbf{x})+\Lambda_{1}q+\Lambda_{2}(\rho-q)\right\} \label{eq:F-max-calc}
\end{equation}
\end{lem}

To complete the result, we recall that $f_{1}(\mu,\mathbf{x})$ is
concave and continuous in $\mathbf{x}$. Thus, we may take its Legendre
transform and apply \eqref{eq:F-max-calc} to find that
\[
f_{1}(\mu,\mathbf{x})=\inf_{\mathbf{\Lambda}}\{F(\mathbf{\Lambda})-\Lambda_{1}q-\Lambda_{2}(\rho-q)\}
\]
as desired. 
\end{proof}

\subsection{Proof of lemmas used  \prettyref{thm:decoupling}}

\begin{proof}[\textbf{\emph{Proof of \prettyref{lem:f1-to-f1-eps-bound}}} ]
We prove the first bound. The second follows by the same argument upon noting that $\Sigma_M^\delta(\rho,q)\subseteq\Sigma_M^\epsilon(\rho,q)$.

Let $\pi:\Sigma_{M}^{\epsilon}(\mathbf{x})\to\Sigma_{M}(\mathbf{x})$
maps a point $\epsilon$ to the closest point in $\Sigma_M$
with respect to the Hamming distance. Note that in the $\ell_{2}$
distance we have $d(\mathbf{x},\pi(\mathbf{x}))\leq C\epsilon N$.
(Throughout this proof $C$ will be a constant that may change from
line to line.) let $\mathbf{Z}(\alpha)=(Z_{i}(\alpha))$, and let
$\tilde{\mathbf{Z}}(\alpha)$ be an independent copy of $\mathbf{Z}$.
Let $\mathbf{Z}_{t}(\alpha,\sigma)=\sqrt{t}\mathbf{Z}(\alpha)\cdot\epsilon+\sqrt{1-t}\tilde{\mathbf{Z}}(\alpha)\cdot\pi(\sigma)$
and consider
\[
\phi(t)=\frac{1}{M}\E\log\sum v_{\alpha}\sum_{\sigma\in\Sigma_{M}^{\epsilon}(\mathbf{x})}e^{\mathbf{Z_{t}}(\alpha,\sigma)},
\]
Since 
\[
\frac{1}{M}\E\partial_{t}\mathbf{Z}_{t}(\alpha,\sigma)\mathbf{Z}_{\mathbf{t}}(\alpha',\sigma')=2\beta^{2}(Q_{\abs{\alpha\wedge\alpha'}}(R(\sigma,\sigma')-R(\pi(\sigma),\pi(\sigma')))
\]
so that $\abs{\phi'}\leq C\epsilon$. 
\[
\phi(0)=\frac{1}{M}\E\log\sum v_{\alpha}\sum_{\sigma\in\Sigma_{M}(\mathbf{x})}card(\pi^{-1}(\sigma))e^{\mathbf{Z}(\alpha)\cdot\sigma}.
\]
Now by a standard counting argument, 
\[
\frac{1}{M}\log\max card(\pi^{-1}(\sigma))\leq\log2-J(1-C\epsilon)\leq C\sqrt{\epsilon}
\]
where $J(x)=\frac{1}{2}(1+x)\log(1+x)+\frac{1}{2}(1-x)\log(1-x),$
where the last inequality holds for $\epsilon$ sufficiently small.
Thus 
\[
f_{1,M}(\mu,\Sigma_{M}(\mathbf{x}))\leq\phi(0)\leq f_{1.M}(\mu,\Sigma_{M}(\mathbf{x}))+C\sqrt{\epsilon}.
\]
On the other hand, by the preceding, $\abs{\phi(1)-\phi(0)}\leq C\epsilon$.
Combining these inequalities and recalling the definition of $\phi(1)$
then yields the claim.
\end{proof}
\begin{proof}[\textbf{\emph{Proof of \prettyref{lem:f1-eps-concave}}}]
 To see this, note that for any $M_{1},M_{2}\geq1$, let $M=M_{1}+M_{2}$
and $\theta=\frac{M_{1}}{M}$. For two points $\mathbf{x},\mathbf{y}\in\cS$
let $\mathbf{z}=\theta\mathbf{x}+(1-\theta)\mathbf{y}$. 
Let $i:\Sigma_{M_{1}}^{\epsilon}(\mathbf{x})\times\Sigma_{M_{2}}^{\epsilon}(\mathbf{y}) \hookrightarrow \Sigma_M$ be the map that takes a point $(u,v)$ to the point $\pi(u,v)$ whose 
 first $\ceil{M_1 x_1}$ coordinates are those of $u$ and next $\ceil{M_2 x_2}$ points are those of $v$, and 
 whose remaining coordinates are given by those of $u$ and $v$ in that order.
 Evidently $i$ is an injection. In fact its image is contained in $\Sigma_M^{\epsilon+1/M}(\mathbf{z})$ (see Lemma ... below)
 Therefore, applying the second inequality from \prettyref{lem:f1-to-f1-eps-bound}, we see that
\[
f_{1,M}(\mu,\Sigma_{M}^{\epsilon}(\mathbf{z}))\geq\lambda f_{1,M_{1}}(\mu,\Sigma_{M}^{\epsilon}(\mathbf{x}))+(1-\lambda)f_{1,M_{2}}(\mu,\Sigma_{M}^{\epsilon}(\mathbf{y}))-\sqrt{1/M}.
\]
Taking $\mathbf{x}=\mathbf{y},$ we see that the sequence $a_M= Mf_{1,M}(\mu,\Sigma_{M}^{\epsilon}(\mathbf{z}))$
satisfies
\[
a_{m_1+m_2} +\phi(m_1+m_2) \geq a_{m_1} +a_{m_2},
\]
for $\phi(t)=\sqrt{t}$. As $\phi$ is increasing and satisfies $\int_1^\infty \phi(t)t^{-2}dt<\infty$, we may apply
the de Brujin-Erdos superadditivity lemma (see below) to find that the desired limit exists. Furthermore,
by taking the limits in the case $\mathbf{x}\neq\mathbf{y}$, we see that the limit is
concave.
\end{proof}
We have used here the following classical result of de Brujin-Erd\"os  \cite[Theorem 23]{deBErdos52}  
(or see \cite[Theorem 1.9.2]{Steele97}).
\begin{thm*}
Let $(a_n)$ be a real sequence and $\phi$ a non-decreasing function with
$\int_1^\infty \phi(t) t^{-2}dt<\infty$. If $(a_n)$ satisfies the superadditivity
criterion
\[
a_{n+m}+\phi(n+m)\geq a_n+a_m \quad \frac{1}{2}n \leq m\leq 2n,
\]
then $\lim_n a_n/n$ converges to $\sup_n a_n/n.$
\end{thm*}

\begin{proof}[\textbf{\emph{Proof of \prettyref{lem:RPC-exp-integral-compute}}}]
 This will follow by calculations involving the Ruelle Probability
Cascades along with a standard covering argument. Consider the functional
\[
F(\mathbf{\Lambda},S)=\frac{1}{M}\E\log\sum_{\alpha}v_{\alpha}\sum_{\sigma\in S}\exp\left(\sum_{i}Z_{i}(\alpha)\sigma_{i}+\Lambda_{1}M_{1}R_{11}^{(1)}+\Lambda_{2}M_{2}R_{11}^{(2)}\right).
\]
 Then
\[
F(\mathbf{\Lambda},\Sigma_{N})=F(\mathbf{\mathbf{\Lambda}})\qquad\text{and}\qquad F(\Lambda,\Sigma_{M}(\rho,q))=f_{1,M}(\mu,\Sigma_{N}(\rho,q))+\Lambda_{1}q+\Lambda_{2}(\rho-q),
\]
where the first equality again follows by \eqref{eq:guerra-to-pde}.

Since $f_{1}$ from \eqref{eq:f1-M-lim} is uniformly continuous on
$\cS$, thus suffices to show that 
\[
F(\mathbf{\Lambda})-\max_{\mathbf{x}\in\cS_{M}}F(\mathbf{\Lambda},\Sigma_{M}(\mathbf{x}))\to0.
\]

To this end, we begin by present a formula for $F(\mathbf{\Lambda};S)$.
Its proof follows from an abstract version of \prettyref{thm:RPC-to-PDE}. Let $(z_{p})_{p\leq r}$
denote a collection of independent centered Gaussian random variables
with 
\[
\E z_{p}^{2}=4\beta^{2}(Q_{p}-Q_{p-1}),
\]
and let $(z_{p,i})_{p\leq r,i\in[N]}$ be i.i.d. copies of this process.
Consider
\[
X_{r}(\mathbf{\Lambda},S)=\log\sum_{\sigma\in S}\exp\left(\sum z_{p,i}\sigma_{i}+\Lambda N_{1}R_{11}^{(1)}+\Lambda N_{2}R_{11}^{(2)}\right).
\]
Define now the sequence of random variables $(X_{\ell})_{\ell=0}^{r}$
by 
\[
X_{\ell}=\frac{1}{\mu_{\ell}}\E_{p}\exp(\mu_{\ell}X_{\ell+1}),
\]
where $\E_{p}$ is expectation in $z_{p}$ and $\mu_{\ell}=\mu([0,Q_{\ell}])$.
As a consequence of \cite[Theorem 2.9]{panchenkobook}, we have that 
\[
F(\mathbf{\Lambda},S)=\frac{1}{M}X_{0}.
\]

Let us now prove the desired limit. By construction, 
\[
\exp(X_{r}(\Lambda,\Sigma_{N}))=\sum_{\mathbf{x}\in\cS_{N}}\exp(X_{r}(\Lambda,\Sigma_{M}(\mathbf{x})).
\]
Since $\mu_{\ell}\leq1$ and $\frac{\mu_{\ell}}{\mu_{\ell+1}}\leq1$,
and $(\sum a_{i})^{s}\leq\sum a_{i}^{s}$ for $s\leq1$ and $a_{i}\geq0$,
we have by an inductive argument that 
\[
\exp(\mu_{p}X_{p})=\E_{p}\exp(\mu_{p}X_{p+1})\leq\sum_{\mathbf{x}\in\cS_{M}}\exp(\mu_{p}X_{p}).
\]
Applying this in the case $p=0$, we see that
\begin{align*}
F(\mathbf{\Lambda})=\frac{1}{M}X_{0}(\mathbf{\Lambda},\Sigma_{M}) & \leq\frac{1}{M\mu_{0}}\log\sum_{\mathbf{x}\in\cS_{M}}\exp \mu_{0}X_{0}(\mathbf{\Lambda},\Sigma_{M}(\mathbf{x}))\\
 & \leq\frac{1}{M\mu_{0}}\log card(\cS_{M})+\max_{\mathbf{x}\in\cS_{M}}F(\mathbf{\Lambda},\Sigma_{M}(\mathbf{x}))
\end{align*}
On the other hand by \eqref{eq:ub} we have that
\[
F(\mathbf{\Lambda})\geq\max_{\mathbf{x}\in\cS_{M}}F(\mathbf{\Lambda},\Sigma_{M}(\mathbf{x})).
\]
 Since $card(\cS_{M})$ is of at most polynomial growth in $M$, the
result follows by the squeeze theorem.
\end{proof}

\section{Statistical feasibility of Estimation: proof of \prettyref{thm:small_signal_main}} 
\label{sec:small_signal}

We prove \prettyref{thm:small_signal_main} in two parts. Recall that the second part was already proved in Section~\ref{sec:ogp}.
It remains to show the first part of Theorem~\ref{thm:small_signal_main}, regarding the impossibility of approximate support recovery. 
In the following, for $a,b\in(0,1)$, let 
\begin{align*}
h(a) &= -a\log a - (1-a)\log (1-a)
\end{align*}

\begin{proof}[\textbf{\emph{Proof of part 1 of Theorem~\ref{thm:small_signal_main}}}]

We will establish the following: for any $0<c<1$, if $\lambda < \frac{c^{3/2}}{2}\sqrt{\frac{1}{\rho}\log(\frac{1}{\rho})}$
then
\[
\limsup_N \frac{1}{N} \sup_{\hat v\in\Sigma_N(\rho)}\inf_{v\in\Sigma_N(\rho)} \E(\hat{v},v) \leq c.
\]
Observe that $\inf_{v \in \Sigma_N(\rho)} \mathbb{E}(\hat{v}, v) \leq \mathbb{E}_{v \sim U(\Sigma_N(\rho))}(\hat{v},v)$, where $\mathbb{E}_{v \sim U(\Sigma_N(\rho))}[\cdot]$ denotes the expectation in $A$ and $v$, where now $v  \sim U(\Sigma_N(\rho))$
instead of being fixed. 
It thus suffices to upper bound
\begin{equation}\label{eq:minimax_lower_bound} 
\limsup_N \frac{1}{N} \sup_{\hat v\in\Sigma_N(\rho)} \E_{v\sim U(\Sigma_N(\rho)} (\hat{v},v) \leq c.
\end{equation}
We do so by an information theoretic argument. 
For two random variables, $X,Y$, let $H(X)$ denote the Shannon entropy, $H(X|Y)$ denote 
the conditional entropy, and let 
\[
I(X;Y) = H(X) - H(X|Y),
\] 
denote their mutual information.

To begin, observe that $v$ and $\hat{v}$ are conditionally independent given $A$, and thus by the data processing 
inequality \cite{cover1991elements}, $I(v; \hat{v}) \leq I(v; A)$. Similarly, note that $A$ and $v$ are conditionally independent given $\frac{\lambda}{N}{vv^{\mathrm{T}}}$, and in fact $v$ is uniquely determined given $\frac{\lambda}{N}{vv^{\mathrm{T}}}$,
and thus $I(v ; A) \leq I(\frac{\lambda}{N}vv^{\mathrm{T}} ; A)$. Combining, we have that $I(v; \hat{v}) \leq I(\frac{\lambda}{N}vv^{\mathrm{T}} ; A)$.

Now, we note that given $\{ \frac{\lambda}{N} v_i v_j : i <j \}$, $\{A_{ij} : i <j\}$ is a product distribution, and thus 
\begin{align}
I \Big( \frac{\lambda}{N}vv^{\mathrm{T}} ; A \Big) \leq \sum_{i < j} I \Big( \frac{\lambda}{N} v_i v_j ; A_{ij} \Big) \leq \frac{N^2}{4} \log \Big( 1 + \frac{\lambda^2 \rho^2}{N}  \Big) \leq \frac{N \lambda^2 \rho^2}{4}. \label{eq:mutualinfo_upper} 
\end{align}
The second inequality above follows using the capacity of a Gaussian additive channel \cite{cover1991elements}.  

On the other hand, we have that 
\begin{align}
I(v;\hat{v}) = H(v) - H(v| \hat{v}) = \log {N \choose N \rho} - H(v| \hat{v}). \nonumber 
\end{align}
Suppose now that $\mathbb{E}_{v \sim U(\Sigma_N(\rho))}[ ( v , \hat{v} )] \geq cN\rho$ for some $c>0$. 
Set $f(\hat{v}) = \mathbb{E}_{v \sim U(\Sigma_N(\rho))}[( v , \hat{v} ) | \hat{v} ]$.  By the  Paley-Zygmund inequality, 
\begin{align}
\mathbb{P}_{v \sim U(\Sigma_N(\rho))}\Big[ f(\hat{v}) \geq \frac{cN\rho}{2} \Big] \geq \frac{1}{4} \frac{ (\mathbb{E}_{v \sim \Sigma_N(\rho) } [f(\hat{v})] )^2 }{ \mathbb{E}_{v \sim \Sigma_N(\rho)}[ (f(\hat{v}))^2] } \geq \frac{c^2}{4} \nonumber 
\end{align} 
which is bounded away from zero, uniformly in $N$ and $\rho$. The inequality above follows from the observations that $0\leq f(\hat{v}) \leq N \rho$, and $\mathbb{E}_{v\sim U(\Sigma_N(\rho))} [f(\hat{v})] \geq c N \rho$.  We set $\mathscr{E} = \{ f(\hat{v}) \geq \frac{cN\rho}{2} \}$ to obtain 
\begin{align}
H(v| \hat{v}) = \mathbb{E}_{v \sim U(\Sigma_N(\rho))} \Big[ \mathbf{1}_{\mathscr{E}} g(\hat{v}) \Big] + \mathbb{E}_{v \sim U(\Sigma_N(\rho))}\Big[ \mathbf{1}_{\mathscr{E}^c} g(\hat{v}) \Big], \nonumber 
\end{align} 
where we use $g(x)$ to denote the entropy of the conditional distribution of $v$ given $\hat{v} = x$. On the event $\mathscr{E}^c$, we use the trivial bound $g(\hat{v}) \leq \log {N \choose N \rho}$. On the event $\mathscr{E}$, we upper bound the conditional entropy as follows. 

Fix $x \in \mathscr{E}$. Using the sub-additivity of conditional entropy, we have, 
\begin{align}
g(x)  &\leq \sum_{i : x_i =1} H(v_i | \hat{v} = x) + \sum_{i : x_i =0} H(v_i | \hat{v}=x) \leq N\rho \log 2 +  N(1-\rho) \cdot \frac{1}{N(1-\rho)}\sum_{i: x_i =0} H(v_i | \hat{v} =x) \nonumber \\
&\leq N \rho \log 2 + N(1-\rho) h\Big( \frac{1}{N(1-\rho)} \sum_{i:x_i=0} \mathbb{E}_{v \sim U(\Sigma_N(\rho))}[v_i | \hat{v}=x] \Big), \nonumber 
\end{align} 
where the last inequality follows by the concavity of the Shannon entropy $h$. 
Observe that on the event $\mathscr{E}$, $\mathbb{E}_{v \sim U(\Sigma_N(\rho))} [\sum_{i: x_i =0} v_i | \hat{v}=x] \leq N \rho (1- c/2)$. For $\rho$ sufficiently small $\rho (1-c/2)/(1-\rho) < 1/2$, and thus maximizing the binary entropy, we obtain the improved upper bound 
\begin{align}
g(x) \leq N \rho \log 2 + N (1-\rho) h\Big( \frac{\rho (1-c/2)}{1-\rho}  \Big). \nonumber  
\end{align}

\noindent
Combining, we have, 
\begin{align}
H(v| \hat{v}) \leq  \mathbb{P}_{v \sim U(\Sigma_N(\rho))}( \mathscr{E}) \Big[ N \rho \log 2 + N (1-\rho) h\Big( \frac{\rho (1-c/2)}{1-\rho}  \Big) \Big] + \mathbb{P}_{v \sim U(\Sigma_N(\rho))}( \mathscr{E}^c) \log {N \choose N \rho}. \nonumber 
\end{align}
Therefore, 
\begin{align}
I(v;\hat{v}) &\geq \log {N \choose N \rho} - \Big[\mathbb{P}_{v \sim U(\Sigma_N(\rho))}( \mathscr{E})\Big( N \rho \log 2 + N (1-\rho) h\Big( \frac{\rho (1-c/2)}{1-\rho}  \Big) \Big)  \nonumber \\
&\qquad\qquad\qquad\qquad+ \mathbb{P}_{v \sim U(\Sigma_N(\rho))}( \mathscr{E}^c)  \log {N \choose N \rho}\Big] \nonumber\\
&\geq  \mathbb{P}_{v \sim U(\Sigma_N(\rho))}( \mathscr{E})  \log {N \choose N \rho} - N\mathbb{P}_{v \sim U(\Sigma_N(\rho))}( \mathscr{E}) \Big[\rho \Big(1-\frac{c}{2} \Big) \log \frac{1}{\rho} + o\Big(\rho \log \frac{1}{\rho} \Big) \Big] \geq \frac{c^3}{16} N\rho\log\frac{1}{\rho}\nonumber 
\end{align}
for $N$ sufficiently large, and $\rho$ sufficiently small. Combining this with \eqref{eq:mutualinfo_upper}, we obtain 
\begin{align} 
\frac{c^3}{16} N \rho \log \frac{1}{\rho} \leq  \frac{N\lambda^2 \rho^2}{4} \implies \lambda^2 \geq \frac{c^3}{4} \cdot\frac{1}{\rho}\log \frac{1}{\rho}. \nonumber 
\end{align}
This contradicts our choice of $\lambda$. Thus the upper bound \eqref{eq:minimax_lower_bound} holds whenever $\lambda \leq c' \sqrt{\frac{1}{\rho} \log \frac{1}{\rho}}$ for some $c' < c^{3/2}/2$. 
This completes the proof. 
\end{proof}

\section{A Simple Rounding Scheme when $\lambda > \frac{1}{\rho}$}
\label{sec:rounding}
In this section, we establish that if the SNR is significantly large, it is algorithmically easy to approximately recover the support of the hidden principal sub-matrix. Specifically, we establish that for $\lambda > 1/\rho$, there exists a simple spectral algorithm which approximately recovers the hidden support. To this end, we start with a simple lemma, which will be critical in the subsequent analysis. 

\begin{lem}
\label{lem:rounding}
Let $(X,Y)$ be jointly distributed random variables, with $X\in \{0,1\}$. Further, assume $\mathbb{P}[X=1]= \rho$, $\mathbb{E}[Y^2] = \rho$, and $\mathbb{E}[XY] \geq \delta \rho$ for some universal constant $\delta>0$. Then 
\begin{align}
\frac{\delta^2}{4} \rho \leq \mathbb{P}\Big[Y >  \frac{\delta}{2} \Big] \leq \Big( 1 + \frac{4}{\delta^2} \Big) \rho. \qquad
\text{ and } \qquad 
\mathbb{P}\Big[X=1 | Y > \frac{\delta}{2} \Big] > \frac{\delta^4}{16}. \nonumber 
\end{align}
\end{lem}

\begin{proof}
First, note that $\mathbb{E}[XY] \geq \delta \rho$ implies $\mathbb{E}[Y|X=1] \geq \delta$. Further, $\mathbb{E}[Y^2] = \rho$ implies $\mathbb{E}[Y^2 | X=1] \leq 1$. 
By the Paley-Zygmund inequality,
\begin{align}
\mathbb{P}\Big[ Y > \frac{\delta}{2} | X=1 \Big] \geq \mathbb{P} \Big[ Y > \frac{1}{2} \mathbb{E}[Y | X=1] | X=1 \Big] \geq \frac{1}{4} \frac{(\mathbb{E}[Y|X=1])^2}{\mathbb{E}[Y^2|X=1]} \geq \frac{\delta^2}{4}. \nonumber 
\end{align} 
This implies the lower bound 
\begin{align}
\mathbb{P}\Big[Y> \frac{\delta}{2} \Big] \geq \rho \mathbb{P}\Big[Y > \frac{\delta}{2} | X=1 \Big] \geq \rho \frac{\delta^2}{4}. \nonumber 
\end{align}
On the other hand, Chebychev inequality implies 
\begin{align}
\mathbb{P}\Big[Y>\frac{\delta}{2} | X=0 \Big] \leq \frac{4}{\delta^2} \mathbb{E}[Y^2| X=0]  \leq \frac{ 4\rho}{(1-\rho)  \delta^2}. \nonumber
\end{align}
Thus, 
\begin{align}
\mathbb{P}\Big[Y > \frac{\delta}{2} \Big] = (1-\rho) \mathbb{P}\Big[Y> \frac{\delta}{2} | X=0 \Big] + \rho \mathbb{P}\Big[Y > \frac{\delta}{2} | X=1 \Big] \leq \rho \Big( 1 + \frac{4}{ \delta^2} \Big). \nonumber 
\end{align}
Finally, using Chebychev's inequality, $\mathbb{P}\Big[Y > \frac{\delta}{2} \Big] \leq \frac{4\rho}{\delta^2}$. Bayes Theorem implies 
\begin{align}
\mathbb{P} \Big[X=1 | Y > \frac{\delta}{2} \Big] = \frac{\rho \mathbb{P}\Big[Y> \frac{\delta}{2} | X=1 \Big]}{ \mathbb{P}\Big[Y > \frac{\delta}{2} \Big]} \geq \frac{\delta^4}{16}. \nonumber 
\end{align}
This completes the proof. 
\end{proof}
\noindent
Armed with Lemma \ref{lem:rounding}, we turn to the proof of Lemma \ref{lemma:rounding}.

\begin{proof}[Proof of Lemma \ref{lemma:rounding}]
Recall that $A = \lambda \rho \frac{v}{\sqrt{N\rho}}( \frac{v}{\sqrt{N\rho}})^T + W$, and thus for $\lambda \rho >1 + \varepsilon$, the celebrated BBP phase transition \cite{Baik2005eigenvalue,benaych2011extreme} implies that there exists a universal constant $\delta := \delta(\varepsilon) >0$ such that w.h.p as $N\to \infty$
\begin{align}
\frac{1}{N} ( v, \hat{v} ) \geq \delta \rho. \nonumber 
\end{align}
In the subsequent discussion, we condition on this good event. 
Consider the two dimensional measure $\mu_N = \frac{1}{N} \sum_{i=1}^{N} \delta_{(v_i, \hat{v}_i)}$. Set $(X,Y) \sim \mu_N$ and note that $(X,Y)$ satisfy the theses of Lemma \ref{lem:rounding}. Lemma \ref{lem:rounding} implies that 
\begin{align}
\frac{\delta^2}{4} \rho N \leq  |\tilde{S}| \leq \Big(1 + \frac{4}{\delta^2} \Big) \rho N. \nonumber 
\end{align}
On the event $|\tilde{S}| < \rho N$, we have, 
\begin{align}
\frac{1}{N} ( v, v_{\hat{S}} ) \geq \frac{\delta^2}{4} \rho \cdot \frac{\delta^4}{16} = \frac{\delta^6}{64} \rho. \nonumber 
\end{align}
On the other hand, if $|\tilde{S}| > N\rho$, 
\begin{align}
( v, v_{\hat{S}} ) \sim \mathrm{Hyp}(|\tilde{S}|, |\tilde{S} \cap \{ i : v_i =1\}|  , N\rho).  \nonumber 
\end{align}
Thus we have, 
\begin{align}
\mathbb{E}[( v , v_{\hat{S}})] = N \rho \frac{|\tilde{S} \cap \{ i : v_i =1\}|}{|\tilde{S}|} = N \rho \mathbb{P}\Big[ X =1 | Y > \frac{\delta}{2} \Big] \geq N \rho \frac{\delta^4}{16}. \nonumber 
\end{align}
Further, direct computation reveals 
\begin{align}
\mathrm{Var}[ ( v, v_{\hat{S}} ) ]  = N\rho \cdot p q \cdot \frac{|\tilde{S}| - N\rho}{|\tilde{S}| -1} ,\nonumber 
\end{align}
where we denote $p := \frac{ |\tilde{S} \cap \{ i : v_i =1\}|}{|\tilde{S}|}= 1 - q$. Upon observing that $p q \leq 1/4$, we have, $\mathrm{Var}[ ( v, \hat{v} ) ] \leq N\rho/4$. Thus by Chebychev inequality, 
\begin{align}
\mathbb{P}\Big[ ( v,  v_{\hat{S}} ) < N\rho \frac{\delta^4}{32} \Big] \leq \frac{N\rho}{\Big( N \rho \frac{\delta^4}{32} \Big)^2} = O\Big(\frac{1}{N}\Big). \nonumber 
\end{align}
This implies that $( v,  v_{\hat{S}} ) \geq N \rho \frac{\delta^4}{32}$ whp over the sampling process, and establishes that the constructed estimator recovers the support approximately. This completes the proof.  
\end{proof}

\appendix

\section{Ruelle Probability Cascades}\label{app:RPC}
For the convenience of the reader, we briefly review here basic properties of Ruelle probability
cascades (RPCs) (sometimes called, Derrida Ruelle Probability Cascades)
used throughout this paper.

\subsection{Construction and basic properties}

Let us begin by recalling the construction of RPCs and some basic properties. 
See, e.g., \cite{panchenkobook} or \cite[Sec. 3.3]{Jag17}.

Fix $r\geq1$ and let $\cA_{r}$ be as in \prettyref{sec:pf_ub}. We label
the vertices of this tree as $\cA_{r}=\mathbb{N}^{0}\cup\mathbb{N}^{1}\cup...\cup\mathbb{N}^{r}$,
where a vertex at depth $k$ has label $\alpha=(\alpha^{1},...,\alpha^{k})$
which corresponds to the root-vertex path, $\emptyset\to\alpha^{1}\to(\alpha^{1},\alpha^{2})\to\cdots\to(\alpha^{1},\ldots,\alpha^{k}).$
As above, we denote this path by $p(\alpha)$. Denote the depth of
a vertex by $\abs{\alpha}$ and let $\partial\cA_{r}$ denote the
leaves of $\cA_{r}$.

For $r\geq1$ and a fixed sequence $0=\mu_{-1}<\mu_{0}<\ldots<\mu_{r}=1$,
we construct the corresponding RPC as follows. Let $m_{\theta}(dx)=\theta x^{-\theta-1}dx$.
For each non-leaf vertex $\alpha\in\cA_{r}\backslash\partial\cA_{r}$,
we assign an independent copy of the Poisson point process $PPP(m_{\mu_{\abs{\alpha}}}(dx))$
arranged in decreasing order, where we assign each child of $\alpha$
the term in the point process of corresponding rank. This yields a
collection $(u_{\alpha})_{\alpha\in\cA_{r}}$ of random variables.
Let $w_{\alpha}=\prod_{\gamma\in p(\alpha)}u_{\gamma}$ and finally
consider the normalized collection $(v_{\alpha})_{\alpha\in\cA_{r}}$
given by 
\[
v_{\alpha}=\frac{w_{\alpha}}{\sum_{\abs{\beta}=\abs{\alpha}}w_{\beta}}.
\]
The \emph{Ruelle probability cascade with parameters $(\mu_{k})_{k=-1}^{r}$
}is the stochastic process $(v_{\alpha})_{\alpha\in\partial\cA_{r}}$. 

It will also be helpful to note the following. Let $\mu\in\cM_1([0,\rho])$ have finite support 
and consider the overlap distribution $\cR(\mu)$ defined as in \eqref{eq:pi-def}.
We note here the following elementary consequence of the definition.
For a proof see, e.g., \cite[Eq. 2.82]{panchenkobook}. 

\begin{lem}\label{lem:overlap-rpc}
Let $\mu\in\cM_1([0,\rho])$ be of finite support and consider $\pi$ as defined in \eqref{eq:pi-def}.
Then $\E \pi(R_{12} = q) = \mu(\{q\})$.
\end{lem}

\subsection{Calculating expectations and Parisi PDEs}

Let us now recall the following well-known result connecting Ruelle
probability cascades to Parisi-type PDEs. (Recall again that we
may view $\mathbb{N}^{r}=\partial\mathcal{A}_{r}$. )
Results of this type appear in different notations throughout the spin glass literature
and are sometimes referred to as consequences of the Bolthausen-Sznitman invariance of
RPCs. The following results are taken from \cite{Arguin2007}.

\begin{thm}[Theorem 6 from \cite{Arguin2007}]\label{thm:RPC-to-PDE} 
Fix $r\geq1,T>0$, and sequences 
\begin{align*}
0 & =q_{0}<q_{1}<\ldots<q_{r}=T\\
0 & =\mu_{-1}<\mu_{0}<\ldots<\mu_{r}=1.
\end{align*}
Let $\psi\in C^{1}([0,T])$ be non-negative  increasing and let
$(g_{\psi}(\alpha))_{\alpha\in\mathbb{N}^{r}}$ denote the centered
Gaussian process with covariance 
\[
\E g_{\psi}(\alpha)g_{\psi}(\beta)=\psi(q_{\abs{\alpha\wedge\beta}}).
\]
Finally, let $(v_{\alpha})_{\alpha\in\N^{r}}$ be a Ruelle probability
cascade with parameters $(\mu_{k})$. Then we have the following.
\begin{enumerate}
\item For any smooth $f$ of at most linear growth we have that
\[
\E\log\sum_{\alpha\in\mathbb{N}^{r}}v_{\alpha}\exp[f(g_{\psi}(\alpha))]=\phi_{\mu}(0,0),
\]
where $\phi$ is the unique solution to 
\begin{align*}
\partial_{t}\phi+\frac{\psi'}{2}\left(\Delta\phi+\mu[0,t](\partial_{x}\phi)^{2}\right) & =0\\
\phi(T,x) & =f(x),
\end{align*}
and $\mu\in\cM_{1}([0,T])$ is given by $\mu(\{q_{k}\})=\mu_{k}-\mu_{k-1}$. 
\item If $(g_{\psi}^{i})_{i=1}^{M}$ are iid copies of $g_{\psi}$ and $(f_{i})$
are of at most linear growth, then 
\[
\E\log\sum_{\alpha}v_{\alpha}\exp\Big(\sum_{i}f_{i}(g_{\psi}^{i}(\alpha)) \Big)=\sum_{i}\E\log\sum v_{\alpha}\exp f_{i}(g_{\psi}(\alpha)). 
\]
\end{enumerate}
\end{thm}

\noindent
We note here the following
corollary which has appeared more or less verbatim in many papers
and follows by applying item 1 in the above with $f(x)=x$
and the Cole-Hopf iteration. 

\begin{cor}[Proposition 7 from \cite{Arguin2007}]
\label{cor:exp-x-term} We have that
\[
\E\log\sum_{\alpha\in\mathbb{N}^{r}}v_{\alpha}\exp[g_{\psi}(\alpha))]=\int_{0}^{T}\psi'(s)s\mu([0,s])ds.
\]
\end{cor}

\noindent
The simplicity of the formula in this case follows from noting that the heat equation with 
initial data $e^x$ is exactly solvable.

\section{Strict convexity}
\label{appendix:convexity} 

To prove strict convexity of $\cP$, let us introduce the following notation.
For the sake of clarity, we make the dependence of $u_{\nu}^{i}$
on $\Lambda_{i}$ explicit by writing $u_{\nu,\Lambda_{i}}(t,x)=u_{\nu}^{i}(t,x)$.
Furthermore, as \eqref{eq:gspde} is invariant under a spatial translation,
 we see that 
\[
u_{\nu,\Lambda}(t,x)=u_{\nu_{0},0}(t,x+\Lambda+\nu(\{1\})).
\]
where $\nu_{0}=\nu-\nu(\{\rho\})\delta_{\rho}$. It will also be helpful
to recall the dynamic programming principle for $u_{\nu,\Lambda_i}(t,x)$
from \cite[Lemma 3.5]{jagannath2017sen}.
\begin{lem*}
For any $(\nu,\lambda)\in\cA\times\R$ of the form $d\nu=mdt+c \delta_{\rho}$, we
have that for any $t<t'\leq\rho$,
\[
u_{\nu,\lambda}(t,x)=\sup_{\alpha\in\cB_{t'}}\E \Big[ u_{\nu,\lambda}(t',X_{t'}^{\alpha})-\int_{t}^{t'}m(s)\alpha_{s}^{2}ds \Big],
\]
where $X^{\alpha}$ solves 
\[
\begin{cases}
dX_{s}=m(s)\alpha_{s}^{2}ds+\sqrt{2}dW_{s}\\
X_{t}=x
\end{cases},
\]
 $W_{s}$ is a standard brownian motion and $\cB_{t'}$ is the
space of bounded stochastic processes that are progressively measurable
with respect to the filtration of $(W_{s})_{s\leq t'}$. Furthermore,
any optimal control $\alpha_{s}^{*}$ satisfies  
\[
m(s)\alpha_{s}^{*}=m(s)\partial_{x}u_{\nu,\lambda}(s,X_{s})\quad a.s.
\]
\end{lem*}
\noindent
The proof will begin with the following observation.
\begin{lem}
For any $\nu\in\cA_{0}$, and any $t\in[0,\rho),$ $u_{\nu,0}(t,x)$
is strictly convex in $x$. 
\end{lem}

\begin{proof}
Fix $x_{0},x_{1}\in\R$
distinct, $\theta\in(0,1)$, and let $x_{\theta}=\theta x_{0}+(1-\theta)x_{0}$.
Let $(X_{s}^{\theta})$ denote the optimal trajectory corresponding
to $u_{\nu,0}(t,x_{\theta})$, and similarly let $\alpha_{s}^{\theta}=\partial_{x}u_{\nu,0}(s,X_{s}^{\theta})$
denote a corresponding optimal control. 

Observe that if we let $G_s=\int_{0}^{s}\sqrt{2}dW_{s_1}+\int_{0}^{s}2m\left(\alpha_{s_1}^{\theta}\right)^{2}ds_1$,
then $X_{s}^{\theta}=G_s+x_{\theta}$. We first claim that the law of $G_\rho$ charges any interval $(a,b)\subseteq \R$.
As it is possible that $\int_0^\rho m^2(s) ds=\infty$, Novikov's condition does not apply, so we cannot apply
Girsanov's theorem directly to $G_\rho$. We circumvent this by a localization argument as follows.

Fix $0<s<\rho$. Since $0\leq\sup |\alpha_{s}^{\theta}|\leq1$
by \eqref{eq:u-x-form}, we have that  $b_t = 2m(t)(\alpha_t^\theta)^2$ has $\sup_{0\leq t \leq  s} | b_{t}| < C(s)$
for some non-random $C(s)>0$. 
By Girsanov's theorem \cite[Lemma 6.4.1]{StroockVaradhan06}  there is a tilt
of the law of $G_s$ such that the law under this tilted measure is Gaussian. 
Thus $\mathbb{P}[G_{s} \in \mathcal{I}] >0$ for any interval $\mathcal{I}$. Now, fix an interval $\mathcal{I}= (a,b)$. 

Note that 
\begin{align}
G_\rho = G_{s'} + \int_{s'}^{\rho} \sqrt{2} dW_{s_1} + \int_{s'}^{\rho} 2m (\alpha^{\theta}_{s_1})^2 ds_1, \nonumber
\end{align}
and thus 
\begin{align}
|G_\rho- G_{s'}| \leq 2 \int_{s'}^{\rho} m ds_1 + \Big| \int_{s'}^{\rho} \sqrt{2} dW_{s_1} \Big|. \nonumber
\end{align}
Further, $\int_{s'}^{\rho} \sqrt{2} dW_{s_1} \sim \mathcal{N}(0, 2(\rho - s'))$. Fix $C'>0$, and let $(c,d) \subset \mathcal{I}$ such that 
\begin{align}
\min\{c-a, d-b\} > 2 \int_{s'}^{\rho} m ds_1 + C' \sqrt{2(\rho - s')}. \nonumber
\end{align}
Such an interval always exists once $\rho-s'$ is sufficiently small. This implies 
\begin{align}
\mathbb{P}[G_{\rho} \in \mathcal{I}] \geq \mathbb{P} \Big[G_{s'} \in (c,d), |G_{\rho}-G_{s'}| \leq 2 \int_{s'}^{\rho} m ds_1 + C' \sqrt{2(\rho - s')} \Big]>0. \nonumber
\end{align}
This, in turn, establishes that for any interval $\mathcal{I}$, $\mathbb{P}[G_{\rho} \in \mathcal{I}] >0$. 

Let $Y=G_{\rho}+x_{1}$ and $Z=G_{\rho}+x_{0}$, then we have that 
\begin{align*}
u_{{\nu,0}}(t,x) & =\E\left[\left(X_{\rho}^{\theta}\right)_{+}-\int_{t}^{\rho}2m\left(\alpha_{s}^{\theta}\right)^{2}ds\right]\\
 & <\theta\E\left[Y_{+}-\int_{t}^{\rho}2m\left(\alpha_{s}^{\theta}\right)^{2}ds\right]
+(1-\theta)\E\left[Z_{+}-\int_{t}^{\rho}2m\left(\alpha_{s}^{\theta}\right)^{2}ds\right]\\
 & \leq\theta u_{\nu,0}(t,x_{0})+(1-\theta)u_{\nu.0}(t,x_{1}),
\end{align*}
where in the second line we use that if $a<0<b$ then $\left(\theta a+(1-\theta)b\right)_{+}<\theta a_{+}+(1-\theta)b_{+}$
and that 
\[
\mathbb{P}(YZ<0)= \mathbb{P}((x_{0}+G_{\rho})(x_{1}+G_{\rho})<0)>0.\qedhere
\]
\end{proof}
\begin{lem}
\label{lem:strict-convexity}
The functional $\cP$ is strictly convex on $\cA_{0}\times\R$.
\end{lem}

\begin{rem}
It will be easy to see from the proof that it is also convex on $\cA\times\R$.
Strict convexity, however, fails on this larger domain due to the invariance of the functional
under the map $(mdt+c\delta_\rho,\Lambda_1,\Lambda_2)\mapsto (mdt,\Lambda_1+2c,\Lambda_2+2c)$.
\end{rem}
\begin{proof}
This proof follows the approach of \cite[Theorem 20]{jagannath2016tobasco}.
Fix $(\nu_{1},\lambda_{1}),(\nu_{2},\lambda_{2})\in\cA_{0}\times\R$,
and $\theta\in[0,1]$, let 
\[
(\nu_{\theta},\lambda_{\theta})=\theta(\nu_{1},\lambda_{1})+(1-\theta)(\nu_{2},\lambda_{2}),
\]
and write $\nu_{1}=m_{1}dt$, $\nu_{2}=m_{2}dt$, $\nu_{\theta}=m_{\theta}dt$
with $m_{\theta}=\theta m_{1}+(1-\theta)m_{2}$. 

Let $X_{s}^{\theta}$ denote the optimal trajectory for the stochastic
control problem for $u_{\nu_{\theta},\lambda_{\theta}}$ and let $\alpha^{\theta}=\partial_{x}u_{\nu_{\theta},\lambda_{\theta}}(s,X_{s}^{\theta})$
the corresponding control. Finally let $Y^{\theta},Z^{\theta}$ solve
the SDEs 
\begin{align*}
dY^{\theta} & =2m_{1}\left(\alpha_{s}^{\theta}\right)^{2}ds+\sqrt{2}dW_{s}\\
dZ^{\theta} & =2m_{2}(\alpha_{s}^{\theta})^{2}ds+\sqrt{2}dW_{s}
\end{align*}
with $Y_{0}=Z_{0}=0.$

Now, fix some $0<t<\rho.$ then by the dynamic programming principle,
\begin{align*}
u_{\nu_{\theta},\lambda_{\theta}}(0,0) & =\E \Big[ u_{\nu_{\theta},\lambda_{\theta}}(t,X_{t}^{\theta})-\int_{0}^{t}m_{\theta}(s)\left(\alpha_{s}^{\theta}\right)^{2}ds \Big].
\end{align*}
Since the equation \eqref{eq:gspde} is invariant under translations of space,
we see that $u_{\nu,\lambda}(t,x)=u_{\nu,0}(t,x+\lambda)$ for any
$(\nu,\lambda)$. Thus we may rewrite the above as 
\begin{align*}
u_{\nu_{\theta},\lambda_{\theta}}(0,0) & =\E u_{\nu_{\theta},0}(t,X_{t}^{\theta}+\lambda_{\theta})-\int_{0}^{t}m_{\theta}(s)\left(\alpha_{s}^{\theta}\right)^{2}ds\\
 & \leq\theta\E\Big(u_{\nu_{\theta},0}(t,Y_{t}+\lambda_{1})-\int_{0}^{t}m_{1}(s)(\alpha_{s}^{\theta})^{2}ds\Big)
 +(1-\theta)\E\Big(u_{\nu_{\theta},0}(t,Z_{t}+\lambda_{2})-\int_{0}^{t}m_{2}(\alpha_{s}^{\theta})^{2}ds\Big)\\
 & \leq\theta u_{\nu_{1},\lambda_{1}}(0,0)+(1-\theta)u_{\nu_{2},\lambda_{2}}(0,0),
\end{align*}
in the first inequality we have used the convexity of $u_{\nu,0}(t,0)$
in space. Note that in fact the first inequality is strict, provided
\[
\mathbb{P}(\left(Y_{t}+\lambda_{1}\right)\neq(Z_{t}+\lambda_{2}))>0.
\]
In particular, it suffices to show that $Var(Y_{t}-Z_{t})+|\lambda_{1}-\lambda_{2}|>0$.
Thus if $\lambda_{1}\neq\lambda_{2}$ we are done. If they are equal
then we know that $m_{1}\neq m_{2}$. In this case, by right continuity
and monotonicity, there must be some $s<\tau<\rho$ such that $m_{1}(t')\neq m_{2}(t')$
on $[s,\tau]$ (that we can take $t<\rho$ follows from the fact that
if $\nu_{1}\neq\nu_{2}$ then $m_{1}$ and $m_{2}$ must differ on
a set of positive lebesgue measure ). In particular, we choose $t=\tau$
from now on. 

Note that by Ito's lemma, our choice of $\alpha_{s}^{\theta}$ is
a martingale, with 
\[
\alpha_{s}^{\theta}-\alpha_{0}^{\theta}=\int_0^{s}\sqrt{2}\partial_{xx}u_{\nu,0}(s_1,X_{s_1}^{\theta})dW_{s_1}.
\]
 Thus by Ito's isometry, if we let $\Delta_{s}=2(m_{1}-m_{2})$
\begin{align*}
Var(Y_{t}-Z_{t}) & =\E\Big(\int_{0}^{t}\Delta_{s}(\alpha_{s}^{\theta}-\alpha_{0}^{\theta}) ds\Big)^{2}
 =\int\int_{[0,t]^{2}}\Delta_{s_1}\Delta_{s_2}K(s_1,s_2)ds_1ds_2,
\end{align*}
where 
\begin{align*}
K(s,t) & =\E(\alpha_{s}-\alpha_{0})(\alpha_{t}-\alpha_{0})
  =\int_{0}^{t\wedge s}2\E\left(\partial_{xx}u_{\nu,0}(s_1,X_{s_1}^{\theta})\right)^{2}ds_1=p(t\wedge s)=p(t)\wedge p(s),
\end{align*}
where $p(t)=\int_{0}^{t}2\E\left(\partial_{xx}u_{\nu,0}(s,X_{s}^{\theta})\right)^{2}ds$.
Notice that since $t<\rho$, we have that $\Delta_{t}\in L^{2}([0,t])$
. Thus to show positivity of the variance, it suffices to show that $K$
is positive definite.

Since $u$ is $C^{2}([0,t+\epsilon]\times\R)$ for some $\epsilon>0$
small enough, and $u$ is strictly convex, we have that $\partial_{xx}u(t,x)>0$
lebesgue a.e. $x$. Thus $p(t)$ is strictly increasing. Thus this
kernel corresponds to a monotone time change of a Brownian motion,
so that it is positive definite. 
\end{proof}

\bibliographystyle{plain}
\bibliography{planted}
\end{document}